\documentclass[a4paper,reqno,10pt]{amsart}

\usepackage{hyperref}
\usepackage{a4wide}

\usepackage{amssymb}
\usepackage{amstext}
\usepackage{amsmath}
\usepackage{amscd}
\usepackage{amsthm}
\usepackage{amsfonts}

\usepackage{graphicx}
\usepackage{latexsym}
\usepackage{mathrsfs}
\usepackage{longtable}

\usepackage{enumitem}
\setlist[enumerate,1]{label={\upshape(\arabic*)}}
\setlist[enumerate,2]{label={\upshape(\alph*)}}

\numberwithin{table}{section}

\usepackage{tikz}
\usetikzlibrary{cd,arrows,matrix,backgrounds,shapes,positioning,calc,decorations.markings,decorations.pathmorphing,decorations.pathreplacing}
\tikzset{blackv/.style={circle,fill=black,inner sep=3pt,outer sep=3pt},
         whitev/.style={circle,fill=white,draw=black,inner sep=3pt,outer sep=3pt},
         blabel/.style={circle,draw=black,inner sep=1.5pt,outer sep=0pt},
         redv/.style={circle,fill=red,inner sep=3pt,outer sep=3pt},
         block/.style={draw,rectangle split,rectangle split horizontal,rectangle split parts=#1},
         symbol/.style={
           draw=none,
           every to/.append style={
             edge node={node [sloped, allow upside down, auto=false]{$#1$}}}}
}

\usepackage{array}
\newcolumntype{C}{>{$}c<{$}}
\newcolumntype{x}[1]{>{\centering\arraybackslash\hspace{0pt}}m{#1}}
\usepackage{makecell}

\newtheorem{theorem}{Theorem}[section]
\newtheorem{theoremi}{Theorem}

\newtheorem{propositioni}[theoremi]{Proposition}

\newtheorem{corollary}[theorem]{Corollary}
\newtheorem{lemma}[theorem]{Lemma}
\newtheorem*{lemma*}{Lemma}
\newtheorem*{theorem*}{Theorem}
\newtheorem{proposition}[theorem]{Proposition}
\newtheorem{definition-proposition}[theorem]{Definition-Proposition}

\newtheorem{question}[theorem]{Question}

\theoremstyle{definition}
\newtheorem{definition}[theorem]{Definition}

\newtheorem{remark}[theorem]{Remark}
\newtheorem{example}[theorem]{Example}

\newtheorem*{ack}{Acknowledgments}

\newcommand{\CC}{\mathcal{C}}

\newcommand{\KK}{\mathcal{K}}
\newcommand{\FF}{\mathcal{F}}
\newcommand{\FFF}{\mathsf{F}}

\newcommand{\HH}{\mathcal{H}}
\renewcommand{\P}{\mathbf{P}}
\newcommand{\I}{\mathbf{I}}

\newcommand{\TT}{\mathcal{T}}
\newcommand{\TTT}{\mathsf{T}}
\newcommand{\UU}{\mathcal{U}}
\newcommand{\WW}{\mathcal{W}}

\newcommand{\Ext}{\operatorname{Ext}\nolimits}

\newcommand{\Hom}{\operatorname{Hom}\nolimits}

\newcommand{\End}{\operatorname{End}\nolimits}

\newcommand{\op}{\operatorname{op}\nolimits}
\newcommand{\RHom}{\mathbf{R}\strut\kern-.2em\operatorname{Hom}\nolimits}

\newcommand{\Image}{\operatorname{Im}\nolimits}
\newcommand{\Kernel}{\operatorname{Ker}\nolimits}
\newcommand{\Cokernel}{\operatorname{Coker}\nolimits}

\newcommand{\coker}{\Cokernel}
\newcommand{\im}{\Image}
\renewcommand{\ker}{\Kernel}

\newcommand{\ov}{\overline}

\newcommand{\ot}{\leftarrow}

\DeclareMathOperator{\moduleCategory}{\mathsf{mod}} \renewcommand{\mod}{\moduleCategory}

\DeclareMathOperator{\proj}{\mathsf{proj}}

\DeclareMathOperator{\tilt}{\mathsf{tilt}}

\DeclareMathOperator{\ccok}{\mathsf{cok}}

\DeclareMathOperator{\Sub}{\mathsf{Sub}}

\DeclareMathOperator{\Fac}{\mathsf{Fac}}

\DeclareMathOperator{\add}{\mathsf{add}}

\DeclareMathOperator{\pd}{\mathsf{pd}}

\newcommand{\iso}{\cong}

\newcommand{\defl}{\twoheadrightarrow}

\newcommand{\equi}{\simeq}

\newcommand{\sst}[1]{\substack{#1}}

\numberwithin{equation}{section}

\begin{document}
\title[IE-closed subcategories of module categories of hereditary algebras]{Image-extension-closed subcategories of module categories of hereditary algebras}
\author[H. Enomoto]{Haruhisa Enomoto}
\address{H. Enomoto: Graduate School of Science, Osaka Metropolitan University, 1-1 Gakuen-cho, Naka-ku, Sakai, Osaka 599-8531, Japan}
\email{henomoto@omu.ac.jp}

\author[A. Sakai]{Arashi Sakai}
\address{A. Sakai: Graduate School of Mathematics, Nagoya University, Chikusa-ku, Nagoya, 464-8602, Japan}
\email{m20019b@math.nagoya-u.ac.jp}

\begin{abstract}
We study IE-closed subcategories of a module category, subcategories which are closed under taking Images and Extensions. We investigate the relation between IE-closed subcategories and torsion pairs, and characterize $\tau$-tilting finite algebras using IE-closed subcategories. For the hereditary case, we show that IE-closed subcategories can be classified by twin rigid modules, pairs of rigid modules satisfying some homological conditions. Moreover, we introduce mutation of twin rigid modules analogously to tilting modules, which gives a way to calculate all twin rigid modules for the representation-finite case.
\end{abstract}

\maketitle
\tableofcontents

\section{Introduction}

Let $\Lambda$ be an artin $R$-algebra over a commutative artinian ring $R$. The aim of this paper is to study \emph{IE-closed subcategories} of $\mod\Lambda$, which are subcategories closed under taking Images and Extensions (Definition \ref{def:basic-def}).
Historically, subcategories \emph{closed under images} were studied by Auslander--Smal{\o} \cite{AS} in the context of functorially finite subcategories, while subcategories \emph{closed under extensions} are quite common in the representation theory of algebras since they can be regarded as exact categories.
Then it is natural to consider the class of \emph{IE-closed} subcategories, which is a large class including all torsion classes, torsion-free classes, wide subcategories, and ICE-closed subcategories introduced in \cite{Eno}.
It turns out that a subcategory of $\mod\Lambda$ is IE-closed if and only if it is an intersection of some torsion class and torsion-free class (Proposition \ref{prop:IE-intersection}), and subcategories of this form naturally appear in various contexts, e.g. the hearts of intervals of torsion pairs \cite{DIRRT, tattar, ES, Eno2} and Leclerc's categorification of the conjectural cluster structure of an open Richardson variety \cite{leclerc}.

First, we study the relation between IE-closed subcategories and torsion pairs in detail, and then show the following result:
\begin{propositioni}[= Proposition \ref{prop:tau-tilt-fin}]
  Let $\Lambda$ be an artin algebra. Then the following conditions are equivalent.
  \begin{enumerate}
    \item $\Lambda$ is $\tau$-tilting finite, that is, there are only finitely many torsion classes in $\mod\Lambda$.
    \item There are only finitely many IE-closed subcategories of $\mod\Lambda$.
    \item Every IE-closed subcategory of $\mod\Lambda$ is functorially finite.
  \end{enumerate}
\end{propositioni}

Next, we aim at classifying IE-closed subcategories by using particular modules similarly to Adachi--Iyama--Reiten's classification of torsion classes \cite{AIR}.
Unfortunately, we have not established such a classification in general. Instead, we focus on hereditary artin algebras, and give a classification in terms of \emph{twin rigid modules}:

\begin{definition}\label{def:twin-rigid}
  Let $\Lambda$ be an artin algebra and $P, I \in \mod\Lambda$. We say that a pair $(P,I)$ is a \emph{twin rigid module} if the following conditions are satisfied.
  \begin{enumerate}
      \item $P$ and $I$ are rigid, that is, $\Ext_{\Lambda}^{1}(P,P)=0$ and $\Ext_{\Lambda}^{1}(I,I)=0$ hold.
      \item There are two short exact sequences
    \begin{equation}\label{eq:twin-rigid-1}
      \begin{tikzcd}
        0 \rar & P \rar & I^{0} \rar & I^{1} \rar & 0,
      \end{tikzcd}
    \end{equation}
    \begin{equation}\label{eq:twin-rigid-2}
      \begin{tikzcd}
        0 \rar & P_{1} \rar & P_{0} \rar & I \rar & 0,
      \end{tikzcd}
    \end{equation}
  with $I^{0}, I^{1}\in\add I$ and $P_{0}, P_{1}\in\add P$.
  \end{enumerate}
\end{definition}
This notion generalizes tilting modules, because $T$ is a tilting module if and only if $(\Lambda, T)$ is a twin rigid module.
Now the following is the main result of this paper.
\begin{theoremi}[= Theorem \ref{thm:main}]\label{thm:a}
  Let $\Lambda$ be a hereditary artin algebra. Then there are mutually inverse bijections between the following two sets.
  \begin{enumerate}
    \item The set of functorially finite IE-closed subcategories $\CC$ of $\mod\Lambda$.
    \item The set of isomorphism classes of basic twin rigid $\Lambda$-modules $(P,I)$.
  \end{enumerate}
  The maps are given as follows:
  \begin{itemize}
    \item $\CC \mapsto (\P(\CC), \I(\CC))$, where $\P(\CC)$ (resp. $\I(\CC)$) is a direct sum of all indecomposable $\Ext$-projective (resp. $\Ext$-injective) objects in $\CC$ up to isomorphism.
    \item $(P,I) \mapsto \Fac P \cap \Sub I$.
  \end{itemize}
\end{theoremi}
We refer the reader to Example \ref{ex:ex} for a concrete example. In addition, the following example is illustrative.
\begin{example}
  For a hereditary artin algebra $\Lambda$, we can consider the following three extreme cases. Note that tilting and cotilting modules coincide in this case.
  \begin{itemize}
    \item $(P, D\Lambda)$ is twin rigid if and only if $P$ is tilting, and the corresponding IE-closed subcategory is $\Fac P$, the usual torsion class corresponding to a tilting module $P$.
    \item $(\Lambda, I)$ is twin rigid if and only if $I$ is cotilting, and the corresponding IE-closed subcategory is $\Sub I$, the usual torsion-free class corresponding to a cotilting module $I$.
    \item $(P,P)$ is twin rigid if and only if $P$ is rigid, and the corresponding IE-closed subcategory is $\add P$.
  \end{itemize}
\end{example}
Unlike Adachi--Iyama--Reiten's classification of torsion classes \cite{AIR}, we have to use both $\Ext$-projectives and $\Ext$-injectives. For example, for a tilting module $P$, the above example yields to twin rigid modules $(P, D\Lambda)$ and $(P,P)$, and the corresponding IE-closed subcategories are $\Fac P$ and $\add P$ respectively, which are different in general.

Then, we observe fundamental properties of twin rigid modules over a hereditary algebra. An important fact is that for a fixed rigid module $P$, twin rigid modules $(P,I)$ bijectively correspond to particular tilting $\End_\Lambda(P)$-modules (Proposition~\ref{prop:tilting}). Taking advantage of this, we introduce \emph{mutation} of twin rigid modules as an analogue of tilting mutation \cite{RS} in classical tilting theory:
\begin{definition}\label{def:mutation}
  Let $(P, X \oplus M)$ be a basic twin rigid module with $X$ indecomposable. If a minimal left $\add M$-approximation $f$ of $X$ is a monomorphism, then we call $(P, \coker f \oplus M)$ a \emph{mutation} of $(P, X \oplus M)$ with respect to $X$.
\end{definition}
We show that a mutation is again a twin rigid module (Proposition \ref{prop:leftmutation}).
Moreover, when $\Lambda$ is of finite representation type, we can find all twin rigid $\Lambda$-modules $(P, I)$ as follows:
\begin{theoremi}[Theorem~\ref{thm:path}]
  Let $\Lambda$ be a hereditary artin algebra of finite representation type and $P$ a rigid $\Lambda$-module. Then any twin rigid module $(P,I)$ can be obtained by iterating mutation from a twin rigid module $(P,P)$.
\end{theoremi}
This gives an algorithm for obtaining all twin rigid modules, and also all IE-closed subcategories by Theorem~\ref{thm:a} (see Example \ref{ex:ex} for an example).

\medskip
\noindent
{\bf Organization.}
This paper is organized as follows.
In Section \ref{sec:ie}, we study basic properties of IE-closed subcategories, and establish a bijection between IE-closed subcategories and twin rigid modules for the hereditary case.
In Section \ref{sec:comandmut}, we focus on twin rigid modules over a hereditary algebra, and introduce mutation and completion of twin rigid modules.
In Section \ref{sec:ex}, we give an example of twin rigid modules and IE-closed subcategories over $k(1 \ot 2 \ot 3)$.
In Appendix \ref{sec:app}, we collect some facts about tilting modules and tilting mutation.

\medskip
\noindent
{\bf Conventions and notation.}

An \emph{artin $R$-algebra} $\Lambda$ is an $R$-algebra $\Lambda$ over a commutative artinian ring $R$ which is finitely generated as an $R$-module. Throughout this paper, we fix a commutative artinian ring $R$, and we simply say that $\Lambda$ is an \emph{artin algebra}.
For an artin algebra $\Lambda$, we denote by $\mod\Lambda$ (resp. $\proj\Lambda$) the category of finitely generated right (resp. finitely generated projective right) $\Lambda$-modules, and $D \colon \mod\Lambda \to \mod\Lambda^{\op}$ denotes the Matlis dual.
We assume that all subcategories are full, additive, and closed under isomorphisms. 
We refer the reader to \cite{ASS,ARS} for the basics of the representation theory of artin algebras.

For $M \in \mod\Lambda$, we will use the following notation.
\begin{itemize}
  \item $|M|$ denotes the number of non-isomorphic indecomposable direct summands of $M$.
  \item $\add M$ denotes the subcategory of $\mod\Lambda$ consisting of direct summands of finite direct sums of $M$.
  \item $\Fac M$ denotes the subcategory of $\mod\Lambda$ consisting of modules $N$ such that there is a surjection $M' \defl N$ with $M' \in \add M$.
  \item $\Sub M$ denotes the subcategory of $\mod\Lambda$ consisting of modules $L$ such that there is an injection $L \hookrightarrow M'$ with $M' \in \add M$.
\end{itemize}

\section{IE-closed subcategories}\label{sec:ie}

In this section, we discuss some finiteness conditions of subcategories and give a proof of Theorem~\ref{thm:a}. Throughout this section, $\Lambda$ denotes an artin algebra, which is not necessarily hereditary.
We begin by collecting basic definitions related to subcategories of $\mod\Lambda$.

\begin{definition}\label{def:basic-def}
  Let $\CC$ be a subcategory of $\mod\Lambda$.
  \begin{enumerate}
    \item $\CC$ is \emph{closed under extensions} if for every short exact sequence
    \[
    \begin{tikzcd}
      0 \rar & L \rar & M \rar & N \rar & 0,
    \end{tikzcd}
    \]
    in $\mod\Lambda$ with $L,N \in \CC$, we have $M \in \CC$.
    \item $\CC$ is \emph{closed under quotients (resp. submodules) in $\mod\Lambda$} if, for every object $C \in \CC$, every quotient (resp. submodule) of $C$ in $\mod\Lambda$ belongs to $\CC$.
    \item $\CC$ is a \emph{torsion class (resp. torsion-free class) in $\mod\Lambda$} if $\CC$ is closed under extensions and quotients in $\mod\Lambda$ (resp. extensions and submodules).
    \item $\CC$ is closed under \emph{images (resp. kernels, cokernels)} if, for every map $\varphi \colon C_1 \to C_2$ with $C_1, C_2 \in \CC$, we have $\im\varphi \in \CC$ (resp. $\ker\varphi\in\CC$, $\coker\varphi\in\CC$).
    \item $\CC$ is an \emph{IE-closed subcategory of $\mod\Lambda$} if $\CC$ is closed under images and extensions.
    \item $\CC$ is a \emph{wide subcategory of $\mod\Lambda$} if $\CC$ is closed under kernels, cokernels, and extensions.
  \end{enumerate}
\end{definition}

\begin{remark}
  If a subcategory $\CC$ of $\mod\Lambda$ is closed under extensions, we can regard $\CC$ as an exact category such that conflations are precisely short exact sequences in $\mod\Lambda$ whose terms are in $\CC$. In what follows, we always regard $\CC$ as an exact category in this case, and we freely use the terminology of exact categories such as exact equivalences, projective objects and injective objects.
\end{remark}

For a collection $\CC$ of $\Lambda$-modules in $\mod\Lambda$, we denote by $\TTT(\CC)$ (resp. $\FFF(\CC)$) the smallest torsion class containing $\CC$ (resp. the smallest torsion-free class containing $\CC$).

\subsection{The general theory of IE-closed subcategories}
In this subsection, we establish basic properties of IE-closed subcategories over an artin algebra.
The following proposition plays an important role in this paper. 

\begin{proposition}\label{prop:IE-intersection}
  For a subcategory $\CC$ of $\mod\Lambda$, the following are equivalent.
  \begin{enumerate}
    \item $\CC$ is an IE-closed subcategory of $\mod\Lambda$.
    \item There exist a torsion class $\TT$ and a torsion-free class $\FF$ in $\mod\Lambda$ satisfying $\CC = \TT \cap \FF$.
  \end{enumerate}
  In this case, $\CC = \TTT(\CC) \cap \FFF(\CC)$ holds.
\end{proposition}
\begin{proof}
  Since every torsion class and torsion-free class is IE-closed, (2) clearly implies (1).
  Conversely, if $\CC$ is IE-closed, then $\CC = \TTT(\CC) \cap \FFF(\CC)$ holds by \cite[Lemma 4.23]{Eno2}. Hence (1) implies (2).
\end{proof}

First, we recall functorial finiteness and $\Ext$-projectives.
\begin{definition}
Let $\CC$ be a subcategory of $\mod\Lambda$ and $M$ an object in $\mod\Lambda$. A morphism $f\colon M\to C$ in $\mod\Lambda$ is a \emph{left $\CC$-approximation} of $M$ if $C$ belongs to $\CC$ and every morphism $f'\colon M\to C'$ with $C'\in\CC$ factors through $f$. Dually, a \emph{right $\CC$-approximation} is defined.
A subcategory $\CC$ of $\mod\Lambda$ is \emph{covariantly finite (resp. contravariantly finite)} if every $M\in\mod\Lambda$ has a left (resp. right) $\CC$-approximation. A subcategory of $\mod\Lambda$ is \emph{functorially finite} if it is covariantly finite and contravariantly finite.
\end{definition}

\begin{definition}
Let $\CC$ be a subcategory of $\mod\Lambda$.
\begin{enumerate}
  \item $P\in\CC$ is \emph{$\Ext$-projective} if it satisfies $\Ext_\Lambda^1(P,C)=0$ for any $C\in\CC$.
  \item \emph{$\CC$ has enough $\Ext$-projectives} if for any $C\in\CC$, there exists a short exact sequence
  \[
  \begin{tikzcd}
    0 \rar & C' \rar & P \rar & C \rar & 0
  \end{tikzcd}
  \]
  such that $P$ is an $\Ext$-projective object in $\CC$ and $C'\in\CC$.
  \item $P$ is an \emph{$\Ext$-progenerator} of $\CC$ if $\CC$ has enough $\Ext$-projectives and $\Ext$-projective objects are precisely objects in $\add P$.
  \item If $\CC$ has an $\Ext$-progenerator, then $\P(\CC)$ denotes a unique basic $\Ext$-progenerator of $\CC$, that is, a direct sum of all indecomposable $\Ext$-projective objects in $\CC$ up to isomorphism.
\end{enumerate}
Dually, the notions for $\Ext$-injectives are defined, and $\I(\CC)$ denotes a unique basic $\Ext$-injective cogenerator of $\CC$ (if exists).
\end{definition}

If $\CC$ is closed under extensions, then $\Ext$-projectives in $\CC$ are precisely projective objects in an exact category $\CC$, and $\CC$ has enough $\Ext$-projectives if and only if $\CC$ has enough projectives as an exact category.
The above notions are related to each other as follows.

\begin{lemma}\label{lem:ie-cov}
  Let $\CC$ be an IE-closed subcategory of $\mod\Lambda$.
  Consider the following conditions.
  \begin{enumerate}
      \item $\CC$ has an $\Ext$-progenerator.
      \item $\CC$ has a finite cover, that is, there exists $M\in\CC$ satisfying $\CC\subseteq\Fac M$.
      \item $\CC$ is covariantly finite.
      \item $\CC$ has enough $\Ext$-projectives.
  \end{enumerate}
  Then the implications {\upshape (1) $\Rightarrow$ (2) $\Leftrightarrow$ (3) $\Rightarrow$ (4)} hold. Moreover, if $\Lambda$ is hereditary, then all the conditions are equivalent.
  \[
  \begin{tikzcd}[column sep = 3cm]
    \textup{(1)} \dar[Rightarrow]
    & \textup{(4)} \lar[Rightarrow, "\textup{if $\Lambda$ is hereditary}"', dashed] \\ 
    \textup{(2)} \rar[Leftrightarrow] & \textup{(3)} \uar[Rightarrow]
  \end{tikzcd}
  \]
\end{lemma}
\begin{proof}
  (1) $\Rightarrow$ (2): An $\Ext$-progenerator $P$ of $\CC$ clearly satisfies $\CC\subseteq\Fac P$.

  (2) $\Leftrightarrow$ (3): This follows from the dual of \cite[Theorem 4.5]{AS}.
  
  (3) $\Rightarrow$ (4): This is the dual of \cite[Proposition 1.1]{Moh}.

  (4) $\Rightarrow$ (1): This follows from \cite[Proposition 3.1 (2)]{Eno}.
  
\end{proof}
The authors do not know whether these conditions are equivalent in general. The implication (4) $\Rightarrow$ (1) can be summarized as follows:
\begin{question}
  Let $\CC$ be an IE-closed subcategory of $\mod\Lambda$ with enough $\Ext$-projectives. Then is the number of indecomposable $\Ext$-projective objects in $\CC$ finite up to isomorphism?
\end{question}

Next, according to \cite{Asa}, we introduce the following finiteness condition.
\begin{definition}
  Let $\CC$ be an IE-closed subcategory of $\mod\Lambda$.
  We say that $\CC$ is \emph{left finite} if $\TTT(\CC)$ is functorially finite. Similarly, we say that $\CC$ is \emph{right finite} if $\FFF(\CC)$ is functorially finite.
\end{definition}

We have the following relation between left finiteness and covariant finiteness. This extends \cite[Propositions 2.28, 4.11]{Asa}, where $\CC$ is assumed to be a wide subcategory.
We remark that these conditions are not equivalent in general, see \cite[Example 4.13]{Asa}.
\begin{lemma}\label{lem:cover_general}
  Let $\CC$ be an IE-closed subcategory of $\mod\Lambda$.
  Consider the following conditions:
  \begin{enumerate}
    \item $\CC$ is left finite.
    \item There exist a torsion class $\TT$ and a torsion-free class $\FF$ such that $\CC=\TT\cap\FF$ and $\TT$ is functorially finite.
    \item $\CC$ is covariantly finite.
  \end{enumerate}
  Then the implications {\upshape (1) $\Rightarrow$ (2) $\Rightarrow$ (3)} hold.
  Moreover, if $\Lambda$ is hereditary, then all the conditions are equivalent.
\end{lemma}
\begin{proof}
  (1) $\Rightarrow$ (2): This follows from Proposition~\ref{prop:IE-intersection}. 
  
  (2) $\Rightarrow$ (3): Since $\TT$ is a functorially finite torsion class, it has an $\Ext$-progenerator. Then \cite[Proposition 5 (iii)]{Kal} implies that $\CC = \TT \cap \FF$ also has an $\Ext$-progenerator. Now (1) $\Rightarrow$ (3) of Lemma \ref{lem:ie-cov} implies that $\CC$ is covariantly finite.
  
  (3) $\Rightarrow$ (1): Suppose that $\Lambda$ is hereditary. By Lemma~\ref{lem:ie-cov}, there is an $\Ext$-progenerator $P$ of $\CC$. Then $P$ is a partial tilting module since $P$ is rigid, so $\Fac P$ is a functorially finite torsion class. We claim $\TTT(\CC)=\Fac P$, which finishes the proof. Indeed, $P \in \CC$ implies $\Fac P \subseteq \TTT(\CC)$. On the other hand, since $\Fac P$ is a torsion class satisfying $\CC \subseteq \Fac P$, we have $\TTT(\CC) \subseteq \Fac P$. 
\end{proof}

Finally, we give a characterization of $\tau$-tilting finite algebras by using IE-closed subcategories. An artin algebra $\Lambda$ is \emph{$\tau$-tilting finite} if there are only finitely many torsion classes in $\mod\Lambda$, or equivalently, every torsion class in $\mod\Lambda$ is functorially finite (see \cite[Theorem 3.8]{DIJ}). The following is an analogue of \cite[Proposition 4.20]{ES}, where the case for ICE-closed subcategories was proved.

\begin{proposition}\label{prop:tau-tilt-fin}
The following conditions are equivalent.
\begin{enumerate}
  \item $\Lambda$ is $\tau$-tilting finite.
  \item The set of IE-closed subcategories of $\mod\Lambda$ is a finite set.
  \item Every IE-closed subcategory of $\mod\Lambda$ is functorially finite.
\end{enumerate}
\end{proposition}
\begin{proof}
Since torsion classes are IE-closed subcategories, both (2) and (3) imply (1).

(1) $\Rightarrow$ (2): Since the set of torsion classes is finite, so is that of torsion-free classes. Then the claim follows easily from Proposition~\ref{prop:IE-intersection}.

(1) $\Rightarrow$ (3): Every IE-closed subcategory is left finite since all torsion classes in $\mod\Lambda$ are functorially finite. Hence every IE-closed subcategory is covariantly finite from Lemma~\ref{lem:cover_general}. Dually, we can show that every IE-closed subcategory is contravariantly finite, so it is functorially finite.
\end{proof}

\subsection{The hereditary case}
In this subsection, \emph{we assume that $\Lambda$ is hereditary} and give a proof of Theorem~\ref{thm:a}. In this case, note that a rigid module is precisely a partial tilting module.
We refer the reader to Appendix \ref{sec:app} for some definitions and facts in tilting theory.

We have the following observation particular to the hereditary case.
\begin{lemma}\cite[Lemma 3.2]{Eno}\label{lem:add}
Let $\Lambda$ be a hereditary artin algebra and $P$ a rigid $\Lambda$-module. Then $\Fac P \cap \Sub P =\add P$ holds, that is, $\add P$ is closed under images.
\end{lemma}

\begin{proposition}\label{prop:pair_eic}
Let $(P,I)$ be a twin rigid module.
Set $\CC = \Fac P \cap \Sub I$. Then the following statements hold.
\begin{enumerate}
  \item $\CC$ is an IE-closed subcategory of $\mod\Lambda$.
  \item $\CC$ is functorially finite in $\mod\Lambda$.
  \item $\CC$ has an $\Ext$-progenerator $P$ and an $\Ext$-injective cogenerator $I$.
\end{enumerate}
\end{proposition}
\begin{proof}

(1)
Since $P$ is a partial tilting module, $\Fac P$ is a torsion class. Dually, $\Sub I$ is a torsion-free class. Thus $\CC=\Fac P\cap\Sub I$ is an IE-closed subcategory. 

(2)
This follows from  the implication (2) $\Rightarrow$ (3) of  Lemma~\ref{lem:cover_general} and its dual. 

(3)
We only prove the statement for $P$. Since $\CC$ is an IE-closed subcategory contained in $\Fac P$, Lemma \ref{lem:ie-cov} implies that $\CC$ has enough $\Ext$-projectives. Thus it suffices to show that an object $C \in \CC$ is $\Ext$-projective in $\CC$ if and only if $C \in \add P$.

Let $C\in\CC$ be an $\Ext$-projective object in $\CC$. Since $C$ is in $\Sub I$, there exists a monomorphism $\iota \colon C \hookrightarrow I^{n}$ for some $n$. Then we have the following diagram with $P_0, P_1 \in \add P$ by \eqref{eq:twin-rigid-2}.
\[
\begin{tikzcd}
  & & & C \dar[hookrightarrow, "\iota"] \ar[dl, dashed, hookrightarrow, "\ov{\iota}"'] \\
  0 \rar & P_{1}^{n} \rar & P_{0}^{n} \rar & I^{n} \rar & 0
\end{tikzcd}
\]
Since $C$ is $\Ext$-projective in $\CC$, we have a map $\ov{\iota} \colon C \to P_{0}^{n}$ which makes the above diagram commute. Then $\ov{\iota}$ is a monomorphism, so we obtain $C\in \Fac P\cap\Sub P$, which implies $C \in \add P$ by Lemma~\ref{lem:add}.

Conversely, we claim that $P$ is $\Ext$-projective in $\CC$. In fact, we have $\Ext_{\Lambda}^{1}(P,\Fac P) = 0$ since $P$ is partial tilting, and thus $\Ext_\Lambda^1(P, \CC) = 0$ holds by $\CC \subseteq \Fac P$. Therefore, every object in $\add P$ is $\Ext$-projective.
\end{proof}

\begin{proposition}\label{prop:eic_pair}
Let $\CC$ be a functorially finite IE-closed subcategory of $\mod\Lambda$. Then the following statements hold.
\begin{enumerate}
  \item $\CC$ has an $\Ext$-progenerator and an $\Ext$-injective cogenerator, which we denote by $P$ and $I$ respectively.
  \item $(P, I)$ is a twin rigid module. 
  \item $\CC = \Fac P \cap\Sub I$ holds.
\end{enumerate}
\end{proposition}

\begin{proof}
(1) 
This follows from Lemma~\ref{lem:ie-cov} and its dual. 
  
(2) 
By definition, $P$ and $I$ are rigid. Since $P$ is an $\Ext$-progenerator of $\CC$, there exists an exact sequence
\[
\begin{tikzcd}
  P_1 \rar["f"] & P_0 \rar & I \rar & 0
\end{tikzcd}
\]
with $P_0, P_1\in\add P$. By Lemma~\ref{lem:add}, we have $\im f\in\add P$, and hence we obtain the sequence \eqref{eq:twin-rigid-2}. By the dual argument, we obtain the sequence \eqref{eq:twin-rigid-1}.

(3)
Every object in $\Fac P \cap\Sub I$ is the image of some map $P' \to I'$ with $P' \in \add P$ and $I' \in \add I$. Since $P', I' \in \CC$ and $\CC$ is closed under images, we obtain $\Fac P \cap \Sub I \subseteq \CC$.
Conversely, every object in $\CC$ belongs to both $\Fac P$ and $\Sub I$ since $P$ is an $\Ext$-progenerator and $I$ is an $\Ext$-injective cogenerator of $\CC$. Thus $\CC = \Fac P \cap \Sub I$ holds.
\end{proof}

We say that a twin rigid module $(P,I)$ is \emph{basic} if both $P$ and $I$ are basic, and a twin rigid module $(P',I')$ is \emph{isomorphic to} $(P,I)$ if $P \iso P'$ and $I \iso I'$ hold.
Now we are ready to give a proof of our main theorem.
\begin{theorem}\label{thm:main}
  The assignments $\CC \mapsto (\P(\CC), \I(\CC))$ and $(P, I) \mapsto \Fac P \cap \Sub I$ give bijections between the following two sets:
  \begin{enumerate}
    \item The set of functorially finite IE-closed subcategories $\CC$ of $\mod\Lambda$.
    \item The set of isomorphism classes of basic twin rigid $\Lambda$-modules $(P,I)$.
  \end{enumerate}
\end{theorem}
\begin{proof}
  If $(P, I)$ is a basic twin rigid module, then we have $P \iso \P(\Fac P \cap \Sub I)$ by Proposition \ref{prop:pair_eic} (3) since a basic $\Ext$-progenerator of $\CC$ is unique up to isomorphism.
  Hence the assertion immediately follows from Propositions~\ref{prop:pair_eic} and \ref{prop:eic_pair}.
\end{proof}

\section{Completion and mutation of twin rigid modules}\label{sec:comandmut}
In this section, we discuss the relation between twin rigid modules and tilting modules for the hereditary case, and introduce completion and mutation of twin rigid modules by using those for tilting modules.

Throughout this section, we fix the following notation.
\begin{itemize}
  \item $\Lambda$ is a \emph{hereditary} artin algebra.
  \item $P \in \mod\Lambda$ is a rigid $\Lambda$-module.
  \item $\Gamma_P := \End_\Lambda(P)$ and $F_P := \Hom_\Lambda(P,-) \colon \mod\Lambda\to \mod\Gamma_P$.
  \item $\ccok P$ is the subcategory of $\mod \Lambda$ consisting of $X \in \mod \Lambda$ such that there is an exact sequence
    \[
    \begin{tikzcd}
      P_1 \rar & P_0 \rar & X \rar & 0
    \end{tikzcd}
    \]
    in $\mod\Lambda$ with $P_0, P_1 \in \add P$.
  \item $\WW_P$ is the smallest wide subcategory of $\mod\Lambda$ containing $P$. Note that $\ccok P \subseteq \WW_P$ holds.
\end{itemize}

Lemma~\ref{lem:add} immediately implies the following alternative description of $\ccok P$:
\begin{lemma}\label{lem:cokp}
  Let $X \in \mod\Lambda$. Then $X \in \ccok P$ if and only if there is a short exact sequence 
  \[
  \begin{tikzcd}
    0 \rar & P_1 \rar & P_0 \rar & X \rar & 0
  \end{tikzcd}
  \]
  in $\mod\Lambda$ with $P_0, P_1 \in \add P$.
\end{lemma}

\begin{remark}
  The definition of twin rigid modules is self-dual, that is, a pair $(P,I)$ of $\Lambda$-modules is a twin rigid $\Lambda$-module if and only if $(DI,DP)$ is a twin rigid $\Lambda^{\op}$-module. Due to this, we can obtain the dual results of this section, which fix $I$ instead of $P$. We omit them since we will not use them.
  \end{remark}

\subsection{Completion of twin rigid modules}

In this subsection, we discuss the relation between twin rigid modules and tilting modules, and introduce completion of twin rigid modules using the Bongartz completion of tilting modules.

We begin with the following observation on $\ccok P$.
\begin{proposition}\label{prop:ice}
  The following assertions hold for a rigid module $P \in \mod\Lambda$.
  \begin{enumerate}
    \item There is a hereditary artin algebra $\Lambda_P$ such that $\WW_P$ is equivalent to $\mod\Lambda_P$ and that $P \in \WW_P$ is a tilting $\Lambda_P$-module under this equivalence.
    \item $\ccok P$ is a torsion class in $\WW_P$, and we have $\ccok P = \Fac_{\WW_P} P$, where $\Fac_{\WW_P}$ denotes $\Fac$ in an abelian category $\WW_P$.
    \item $\ccok P$ is closed under images, cokernels, and extensions.
    \item $\Sub(DP)$ is a torsion-free class in $\mod\Gamma_P$, and $F_P$ induces an exact equivalence between $\ccok P$ and $\Sub(DP)$.
    \item For every $X \in \ccok P$, we have $\pd_{\Gamma_P} F_P X \leq 1$.
  \end{enumerate}
\end{proposition}
\begin{proof}
  (1)--(3) These follow from \cite[Proposition 5.1, Corollary 5.2]{Eno}.

  (4)
  Under the equivalence in (1), this follows from the Brenner--Butler theorem (Proposition \ref{prop:tilt-comp-BB} (2)).

  (5)
  Lemma \ref{lem:cokp} implies that there is a short exact sequence
  \[
    \begin{tikzcd}
      0 \rar & P_1 \rar & P_0 \rar & X \rar & 0
    \end{tikzcd}
  \]
  in $\mod\Lambda$ with $P_0, P_1 \in \add P$. Since this is a short exact sequence in $\ccok P$ and $F_P$ is exact on $\ccok P$, by applying $F_P$, we obtain the following short exact sequence in $\mod\Gamma_P$:
  \[
    \begin{tikzcd}
      0 \rar & F_P P_1 \rar & F_P P_0 \rar & F_P X \rar & 0.
    \end{tikzcd}
  \]
  Since $F_P P = \Gamma$, we have $F_P P_0, F_P P_1 \in \proj\Gamma$. Therefore, $\pd_{\Gamma_P} F_P X \leq 1$ holds.
\end{proof}

Using this, we can obtain the following relation between twin rigid modules $(P,I)$ and tilting $\Gamma_P$-modules.

\begin{proposition}\label{prop:tilting}
  The following statements hold.
  \begin{enumerate}
    \item Let $I \in \mod\Lambda$. Then $(P,I)$ is a twin rigid module if and only if $I \in \ccok P$ holds and $F_P I$ is a tilting $\Gamma_P$-module. In particular, $|P| = |I|$ holds in this case.
    \item The exact equivalence $F_P \colon \ccok P \equi \Sub (DP)$ gives a bijection between the following sets up to isomorphism.
    \begin{enumerate}
      \item The set of $\Lambda$-modules $I$ such that $(P,I)$ is twin rigid.
      \item The set of tilting $\Gamma_P$-modules $T$ contained in $\Sub(DP)$.
    \end{enumerate}
  \end{enumerate}
\end{proposition}
\begin{proof}
  (1)
  Recall that we have an exact equivalence $F_P \colon \ccok P \equi \Sub (DP) \subseteq \mod\Gamma_P$ by Proposition \ref{prop:ice} (4), which restricts to an equivalence $\add P \equi \proj \Gamma_P$ with $F_P P = \Gamma_P$. Therefore, the following statements hold for $I \in \ccok P$.
  \begin{itemize}
    \item $I$ is rigid if and only if $F_P I$ is rigid.
    \item There exists a short exact sequence of the form \eqref{eq:twin-rigid-1} with $I^0, I^1 \in \add I$ if and only if there exists a short exact sequence of the following form in $\mod\Gamma_P$:
      \[
      \begin{tikzcd}
        0 \rar &  \Gamma_P \rar & F_P I^{0} \rar & F_P I^{1} \rar & 0.
      \end{tikzcd}
      \]
  \end{itemize}
  Also, for every $I \in \ccok P$, we have $\pd_{\Gamma_P} F_P I \leq 1$ by Proposition \ref{prop:ice} (5), and a short exact sequence of the form \eqref{eq:twin-rigid-2} always exists. Hence (1) follows from the above facts.
  Moreover, in this case, we have $|P| = |\Gamma_P| = |F_P I| = |I|$ holds since $F_P \colon \ccok P \equi \Sub (DP)$ is an equivalence and $F_P I$ is a tilting $\Gamma_P$-module, see Proposition \ref{prop:tilt-comp-BB} (1).

  (2) By (1), we obtain a map $I \mapsto F_P I$ from (a) to (b). Conversely, for $T$ in (b), there is some $I \in \ccok P$ with $FI \iso T$ by the equivalence $F_P \colon \ccok P \equi \Sub (DP)$. Then (1) implies that $(P,I)$ is twin rigid.
\end{proof}
Using this, we can deduce various properties of twin rigid modules from those of tilting modules.
First, we have the following analogue of the Bongartz completion.
\begin{theorem}\label{thm:completion}
  Let $P$ be a rigid $\Lambda$-modules. Suppose that $I$ is rigid and $I \in \ccok P$. Then there exists $I'$ such that $(P,I\oplus I')$ is a twin rigid module.
\end{theorem}
\begin{proof}
  An exact equivalence $F_P \colon \ccok P \equi \Sub (DP)$ implies that $F_P I$ is a rigid $\Gamma_P$-module. Moreover, we have $\pd_{\Gamma_P} F_P I \leq 1$ by Proposition \ref{prop:ice} (5). Therefore, $F_P I$ is a partial tilting $\Gamma_P$-module.
  Therefore, Proposition \ref{prop:tilt-comp-BB} (1) implies that there is some $E \in \mod \Gamma_P$ such that $F_P I \oplus E$ is a tilting $\Gamma_P$-module and that there is a short exact sequence
  \[
    \begin{tikzcd}
      0 \rar & \Gamma_P \rar & E \rar & M_0 \rar & 0
    \end{tikzcd}
  \]
  with $M_0 \in \add (F_P I)$.
  Then since $F_P I$ and $\Gamma_P$ belong to $\Sub (DP)$ and $\Sub(DP)$ is extension-closed, we have $E \in \Sub(DP)$, and hence $F_P I \oplus E$ is a tilting $\Gamma_P$-module contained in $\Sub (DP)$.
  Therefore, Proposition \ref{prop:tilting} implies that there is some $I' \in \ccok P$ such that $(P, I \oplus I')$ is twin rigid.
\end{proof}

As a corollary, we have the following numerical criterion of twin rigid modules.
\begin{proposition}\label{prop:rigidnumber}
  Let $P$ be a rigid $\Lambda$-module and $I \in \mod\Lambda$. Then $(P,I)$ is a twin rigid module if and only if the following conditions are satisfied:
  \begin{enumerate}
    \item $I$ is rigid.
    \item $I \in \ccok P$ holds, or equivalently, there exists a short exact sequence
    \[
    \begin{tikzcd}
      0 \rar & P_1 \rar & P_0 \rar & I \rar & 0.
    \end{tikzcd}
    \]
    \item $|I| = |P|$ holds, or equivalently, $|I| \geq |P|$ holds.
  \end{enumerate}
\end{proposition}
\begin{proof}
  Suppose that $(P,I)$ is twin rigid. Then clearly (1) and (2) hold by definition, and (3) follows from Proposition~\ref{prop:tilting} (1).

  Conversely, suppose that $I$ satisfies (1)--(3).
  Then by Theorem \ref{thm:completion}, there is some $I' \in \mod\Lambda$ such that $(P,I \oplus I')$ is twin rigid. Thus we have $|I \oplus I'| \geq |I| \geq |P|  = |I \oplus I'|$, where the last equality follows from Proposition \ref{prop:tilting} (1). Thus we obtain $|I \oplus I'| = |I|$, which implies $I' \in \add I$. Therefore, $(P,I)$ is twin rigid.
\end{proof}
Also we have the following restriction on the number of ``complements'' of twin rigid modules for a fixed $P$.
\begin{proposition}\label{prop:twinatmosttwo}
Let $M\in\ccok P$. If $|M|=|P|-1$ holds, then there exist at most two non-isomorphic indecomposable $\Lambda$-modules $X$ such that $(P,M \oplus X)$ is a twin rigid module.
\end{proposition}
\begin{proof}
This follows easily from Proposition~\ref{prop:tilting} and Theorem~\ref{thm:exists-exchange-seq} (1).
\end{proof}

\begin{remark}
  In \cite{AS2}, Auslander and Smal{\o} raised the following questions: for an artin algebra $\Lambda$, does every functorially finite subcategory $\CC$ of $\mod\Lambda$ which is closed under extensions have an $\Ext$-progenerator $P$ and an $\Ext$-injective cogenerator $I$? Also, does the equality $|P|=|I|$ holds? It is well-known that these questions are true when $\CC$ is either a torsion class, a torsion-free class, or a wide subcategory. Theorem~\ref{thm:main} and Proposition \ref{prop:tilting} (1) imply that the questions are true when $\CC$ is an IE-closed subcategory and $\Lambda$ is hereditary.
\end{remark}

\subsection{Mutation of twin rigid modules}

In this subsection, we introduce \emph{mutation} of twin rigid modules, which enables us to calculate all twin rigid modules when $\Lambda$ is of finite representation type. We begin with introducing a quiver consisting of twin rigid modules, which is inspired by the quiver of tilting modules (see Theorem~\ref{thm:hasse}).
Recall that we have fixed a rigid module $P$ over a hereditary algebra $\Lambda$.
In addition, \emph{we assume that $P$ is basic} in this subsection.

\begin{definition}
We define the quiver $\KK_P$ as follows.
\begin{itemize}
  \item The vertex set of $\KK_P$ is the set of isomorphism classes of basic twin rigid modules $(P,I)$. 
  \item For $(P,I_1), (P,I_2)\in\KK_P$, we draw an arrow $(P,I_1)\to (P,I_2)$ in $\KK_P$ if we have decompositions $I_1=X\oplus M$ and $I_2=Y\oplus M$ with $X$ and $Y$ indecomposable such that there is a short exact sequence
    \begin{equation}\label{eq:exchange-twin}
      \begin{tikzcd}
        0 \rar & X \rar & M' \rar & Y \rar & 0,
      \end{tikzcd}
    \end{equation}
with $M'\in\add M$. Note that $I_1 \not \iso I_2$ holds in this case, because otherwise the above sequence would split.
\end{itemize}
We call the sequence \eqref{eq:exchange-twin} an \emph{exchange sequence} of an arrow $(P,I_1)\to (P,I_2)$.
\end{definition}

An arrow in $\KK_P$ induces a minimal inclusion of the corresponding IE-closed subcategories.

\begin{proposition}
Let $(P,I_1)\to (P,I_2)$ be an arrow in $\KK_P$. Set $\CC_i=\Fac P\cap\Sub I_i$ for $i=1,2$. Then $\CC_1\subsetneq\CC_2$ holds and there is no functorially finite IE-closed subcategory $\CC$ which satisfies $\CC_1\subsetneq\CC\subsetneq\CC_2$.
\end{proposition}
\begin{proof}
By definition, we have decompositions $I_1=X\oplus M$ and $I_2=Y\oplus M$ and an exchange sequence \eqref{eq:exchange-twin}, which implies $X \in \Sub M \subseteq \Sub I_2$. Thus $\Sub I_1\subseteq\Sub I_2$ holds, so we have $\CC_1\subseteq\CC_2$.
By Theorem~\ref{thm:a}, we have $\CC_1\subsetneq\CC_2$.

Suppose that there exists a functorially finite IE-closed subcategory $\CC$ of $\mod\Lambda$ with $\CC_1\subseteq \CC\subseteq \CC_2$.
Since $P,M\in\CC_1$, we have $P,M\in\CC$. Since $P$ is $\Ext$-projective and $M$ is $\Ext$-injective in $\CC_2$, they are so in $\CC$. Moreover, since $P$ is an $\Ext$-progenerator of $\CC_2$, we have  $\CC \subseteq \CC_2 \subseteq \ccok P$. Thus Lemma \ref{lem:cokp} shows that $\CC$ has an $\Ext$-progenerator $P$, so $\P(\CC) = P$. In addition, $\I(\CC)$ contains $M$ as a direct summand. Since $|M| = |P| - 1$, Proposition~\ref{prop:twinatmosttwo} shows that a twin rigid module $(P,\I(\CC))$ coincides with either $(P,I_1)$ or $(P,I_2)$. Thus Theorem~\ref{thm:a} implies $\CC = \CC_1$ or $\CC = \CC_2$.
\end{proof}
\begin{remark}
  If there is a tilting exchange sequence from $T_1$ to $T_2$, then we have the minimal inclusion of torsion classes $\Fac T_1 \supsetneq \Fac T_2$ (Theorem \ref{thm:hasse}), and this inclusion is opposite to the above one. This is because we consider $\Sub I$ and $\Sub I'$, not $\Fac I$ and $\Fac I'$.
\end{remark}

We define mutation of twin rigid modules in Definition~\ref{def:mutation}, which gives a way to obtain a new twin rigid module from the old one, which is an analogue of (left) mutation of tilting modules (Proposition~\ref{prop:tilt-mutation}). A mutation of a twin rigid module $(P,I)$ is also a twin rigid module and gives an adjacent vertex in $\KK_P$:

\begin{proposition}\label{prop:leftmutation}
Let $(P,X\oplus M)$ be a basic twin rigid module with $X$ indecomposable. If a minimal left $\add M$-approximation $f\colon X\to M'$ is a monomorphism, then the following hold.
\begin{enumerate}
  \item $(P, \Cokernel f\oplus M)$ is a twin rigid module. 
  \item There is an arrow $(P,X\oplus M)\to (P,\Cokernel f\oplus M)$ in $\KK_P$.
\end{enumerate}
\end{proposition}
\begin{proof}
(1) Consider a short exact sequence
  \begin{equation}\label{eq:approximation}
        \begin{tikzcd}
        0 \rar & X \rar["f"] & M' \rar & \Cokernel f \rar & 0.
      \end{tikzcd}
  \end{equation}
Since $X$ and $M'$ are in $\ccok P$, so is $\Cokernel f$ because $\ccok P$ is closed under cokernels (Proposition \ref{prop:ice} (3)). Applying $F_P$ to the above exact sequence, we obtain a short exact sequence
  \[
  \begin{tikzcd}
    0 \rar & F_PX \rar["F_Pf"] & F_PM' \rar & F_P(\Cokernel f) \rar & 0
  \end{tikzcd}
  \]
in $\mod\Gamma_P$ since $F_P$ is exact on $\ccok P$. Then $F_Pf$ is a minimal left $\add F_PM$-approximation since $F_P$ is fully faithful on $\ccok P$ by Proposition~\ref{prop:ice} (4). Proposition~\ref{prop:tilt-mutation}, implies that $F_P(\Cokernel f\oplus M)$ is a tilting $\Gamma_P$-module with $F_P (\coker f)$ indecomposable. By Proposition~\ref{prop:tilting} (1), we obtain the desired result.

(2) Note that $\coker f$ is indecomposable since so is $F_P \coker f$ and $F_P$ is fully faithful on $\ccok P$. Then \eqref{eq:approximation} gives an exchange sequence of $(P,X\oplus M)\to (P,\Cokernel f\oplus M)$.
\end{proof}

As a corollary, we can interpret arrows in $\KK_P$ using mutation.
\begin{corollary}
  Let $(P,I)$ and $(P,I')$ be basic twin rigid modules. Then there is an arrow $(P,I) \to (P,I')$ in $\KK_P$ if and only if there is some indecomposable direct summand $X$ of $I$ such that $(P,I')$ is a mutation of $(P,I)$ with respect to $X$.
\end{corollary}
\begin{proof}
  The ``if'' part is Proposition \ref{prop:leftmutation}. To show the ``only if'' part, suppose that there is an arrow $(P,I) \to (P,I')$. Then there is a short exact sequence \eqref{eq:exchange-twin}. It follows that the left map in \eqref{eq:exchange-twin} is a minimal left $\add M$-approximation. Indeed, we can show that it is a left $\add M$-approximation since $\Ext_\Lambda^1(Y,M) = 0$, and it is left minimal since this sequence does not split and $Y$ is indecomposable. Therefore, $(P,I')$ is a mutation of $(P,I)$ with respect to $X$.
\end{proof}

In the rest of this subsection, we will show the following result:
\begin{theorem}\label{thm:path}
Assume that $\Lambda$ is of finite representation type. 
For any basic twin rigid module $(P,I)$, there exists a path $(P,P)\to\cdots\to (P,I)$ in $\KK_P$, that is, $(P,I)$ can be obtained by iterating mutation from $(P,P)$.
\end{theorem}

To give a proof, we need the following  two results related to the Hasse quiver of the poset $\tilt\Gamma_P$ (see Definition~\ref{def:tilt}).

\begin{lemma}\label{lem:subdp}
 Let $T'\to T$ be an arrow in $\HH(\tilt\Gamma_P)$. If $T$ belongs to $\Sub(DP)$, then so does $T'$.
\end{lemma}
\begin{proof}
By Theorem~\ref{thm:hasse}, there exists a tilting exchange sequence from $T'$ to $T$. This sequence implies $T' \in \Sub T$, and we have $T \in \Sub (DP)$ by assumption.
Therefore, we obtain $T'\in \Sub T \subseteq \Sub(DP)$. 
\end{proof}

\begin{lemma}\label{lem:fullsub}
Let $(P,I)$ and $(P,I')$ be basic twin rigid modules. There exists an arrow $(P,I)\to(P,I')$ in $\KK_P$ if and only if there exists an arrow $F_PI\to F_PI'$ in $\HH(\tilt\Gamma_P)$. In particular, $F_P$ induces a quiver isomorphism between $\KK_P$ and the full subquiver of $\HH(\tilt\Gamma_P)$ consisting of vertices in $\Sub(DP)$. 
\end{lemma}
\begin{proof}
Suppose that there is an arrow $(P,I)\to(P,I')$. By definition, we have $I=X\oplus M$ and $I'=Y\oplus M$ and there is an exchange sequence \eqref{eq:exchange-twin}. Applying $F_P$ to this, we obtain a tilting exchange sequence from $F_PI$ to $F_PI'$. By Theorem~\ref{thm:hasse}, we have an arrow $F_PI\to F_PI'$ in $\HH(\tilt\Gamma_P)$.

Conversely, assume that there exists an arrow $F_PI\to F_PI'$ in $\HH(\tilt\Gamma_P)$. By Theorem~\ref{thm:hasse}, there exists a tilting exchange sequence from $F_P I$ to $F_P I'$. Since $F_P$ is fully faithful on $\ccok P$, we can write the tilting exchange sequence as follows:
  \[
  \begin{tikzcd}
    0 \rar & F_PX \rar & F_PM' \rar & F_PY \rar & 0,
  \end{tikzcd}
  \]
where $I = X \oplus M$, $I' = Y \oplus M$, $M' \in \add M$, and $X$ and $Y$ are indecomposable. Since $F_P$ induces an exact equivalence $\ccok P \equi \Sub (DP)$,
we obtain an exchange sequence of $(P,I)\to(P,I')$.

The latter part follows from the former part and Proposition~\ref{prop:tilting} (2).
\end{proof}

\begin{proof}[Proof of Theorem~\ref{thm:path}]
It follows from \cite[VIII. Lemma 3.2 (c)]{ASS} that $\Gamma_P$ is of finite representation type. Thus $\HH(\tilt\Gamma_P)$ is a finite quiver, so there exists a path $\Gamma_P\to\cdots\to F_PI$ in $\HH(\tilt\Gamma_P)$ since $\Gamma_P \geq F_P I$ holds in $\tilt \Gamma_P$. Since $F_P I$ belongs to $\Sub (DP)$, by applying Lemma~\ref{lem:subdp} repeatedly, we can show that all the vertices appearing in this path belong to $\Sub(DP)$. By Lemma~\ref{lem:fullsub}, we obtain the desired result. 
\end{proof}

\section{Examples}\label{sec:ex}

In this section, we give some examples and remarks of IE-closed subcategories and twin rigid modules over a hereditary algebra. We denote by $k$ a field.

By Theorem \ref{thm:path}, we can obtain all twin rigid $\Lambda$-modules if $\Lambda$ is a hereditary algebra of finite representation type as follows: First, we can obtain all tilting $\Lambda$-modules $T$ by iterating mutation from $(\Lambda, \Lambda)$, since $T$ is tilting if and only if $(\Lambda, T)$ is twin rigid.
Then rigid $\Lambda$-modules are precisely direct summands of tilting modules, so we obtain all rigid $\Lambda$-modules.
Finally, for each rigid $\Lambda$-module $P$, by applying mutation from $(P,P)$ repeatedly, we obtain all twin rigid modules.

\begin{example}\label{ex:ex}
Let $Q$ be a quiver $1\leftarrow 2\leftarrow 3$. The Auslander-Reiten quiver of $\mod kQ$ is as follows:
      
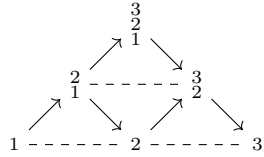
\[
      \begin{tikzpicture}[scale = 0.8]
        \node (1) at (0,0) {$\sst{1}$};
        \node (2) at (2,0) {$\sst{2}$};
        \node (12) at (1,1) {$\sst{2 \\ 1}$};
        \node (123) at (2,2) {$\sst{3 \\ 2 \\ 1}$};
        \node (23) at (3,1) {$\sst{3 \\ 2}$};
        \node (3) at (4,0) {$\sst{3}$};

        \draw[->] (1) -- (12);
        \draw[->] (12) -- (2);
        \draw[->] (12) -- (123);
        \draw[->] (123) -- (23);
        \draw[->] (2) -- (23);
        \draw[->] (23) -- (3);
        \draw[dashed] (1) -- (2);
        \draw[dashed] (12) -- (23);
        \draw[dashed] (2) -- (3);
      \end{tikzpicture}
      \]
We can obtain all twin rigid modules by using Theorem \ref{thm:path} and the above argument, and
Table~\ref{tab:ie} is the complete list of IE-closed subcategories $\CC$ of $\mod kQ$ and the corresponding twin rigid modules $(P,I)$ in Theorem~\ref{thm:a}. The black vertices correspond to the indecomposable modules in $\CC$ and the indecomposable summands of $P$ and $I$ respectively.

  \begin{longtable}[c]{x{3cm}|x{3cm}||x{3cm}|x{3cm}}
      \hline
      $\CC$ & $(P,I)$ & $\CC$ & $(P,I)$
      \\ \hline \hline
      \endfirsthead
      \hline
      $\CC$ &$(P,I)$ & $\CC$ & $(P,I)$ 
      \\ \hline \hline
      \endhead
      \begin{tikzpicture}[scale=0.2, every node/.style={scale=0.5}]
        \node (1) at (0,0) [blackv] {};
        \node (2) at (2,0) [whitev] {};
        \node (12) at (1,1) [blackv] {};
        \node (123) at (2,2) [blackv] {};
        \node (23) at (3,1) [whitev] {};
        \node (3) at (4,0) [whitev] {};

      \end{tikzpicture}
       & 
      \begin{tikzpicture}[scale=0.2, every node/.style={scale=0.5}]
        \node (1) at (0,0) [blackv] {};
        \node (2) at (2,0) [whitev] {};
        \node (12) at (1,1) [blackv] {};
        \node (123) at (2,2) [blackv] {};
        \node (23) at (3,1) [whitev] {};
        \node (3) at (4,0) [whitev] {};

      \end{tikzpicture},
      \begin{tikzpicture}[scale=0.2, every node/.style={scale=0.5}]
        \node (1) at (0,0) [blackv] {};
        \node (2) at (2,0) [whitev] {};
        \node (12) at (1,1) [blackv] {};
        \node (123) at (2,2) [blackv] {};
        \node (23) at (3,1) [whitev] {};
        \node (3) at (4,0) [whitev] {};

      \end{tikzpicture}
      &
      \begin{tikzpicture}[scale=0.2, every node/.style={scale=0.5}]
        \node (1) at (0,0) [blackv] {};
        \node (2) at (2,0) [blackv] {};
        \node (12) at (1,1) [blackv] {};
        \node (123) at (2,2) [blackv] {};
        \node (23) at (3,1) [whitev] {};
        \node (3) at (4,0) [whitev] {};

      \end{tikzpicture}
      & 
      \begin{tikzpicture}[scale=0.2, every node/.style={scale=0.5}]
        \node (1) at (0,0) [blackv] {};
        \node (2) at (2,0) [whitev] {};
        \node (12) at (1,1) [blackv] {};
        \node (123) at (2,2) [blackv] {};
        \node (23) at (3,1) [whitev] {};
        \node (3) at (4,0) [whitev] {};

      \end{tikzpicture},
      \begin{tikzpicture}[scale=0.2, every node/.style={scale=0.5}]
        \node (1) at (0,0) [whitev] {};
        \node (2) at (2,0) [blackv] {};
        \node (12) at (1,1) [blackv] {};
        \node (123) at (2,2) [blackv] {};
        \node (23) at (3,1) [whitev] {};
        \node (3) at (4,0) [whitev] {};

      \end{tikzpicture}
      \\ \hline
      \begin{tikzpicture}[scale=0.2, every node/.style={scale=0.5}]
        \node (1) at (0,0) [blackv] {};
        \node (2) at (2,0) [blackv] {};
        \node (12) at (1,1) [blackv] {};
        \node (123) at (2,2) [blackv] {};
        \node (23) at (3,1) [blackv] {};
        \node (3) at (4,0) [whitev] {};

      \end{tikzpicture}
       & 
      \begin{tikzpicture}[scale=0.2, every node/.style={scale=0.5}]
        \node (1) at (0,0) [blackv] {};
        \node (2) at (2,0) [whitev] {};
        \node (12) at (1,1) [blackv] {};
        \node (123) at (2,2) [blackv] {};
        \node (23) at (3,1) [whitev] {};
        \node (3) at (4,0) [whitev] {};

      \end{tikzpicture},
      \begin{tikzpicture}[scale=0.2, every node/.style={scale=0.5}]
        \node (1) at (0,0) [whitev] {};
        \node (2) at (2,0) [blackv] {};
        \node (12) at (1,1) [whitev] {};
        \node (123) at (2,2) [blackv] {};
        \node (23) at (3,1) [blackv] {};
        \node (3) at (4,0) [whitev] {};

      \end{tikzpicture}
      &
      \begin{tikzpicture}[scale=0.2, every node/.style={scale=0.5}]
        \node (1) at (0,0) [blackv] {};
        \node (2) at (2,0) [whitev] {};
        \node (12) at (1,1) [blackv] {};
        \node (123) at (2,2) [blackv] {};
        \node (23) at (3,1) [whitev] {};
        \node (3) at (4,0) [blackv] {};

      \end{tikzpicture}
       & 
      \begin{tikzpicture}[scale=0.2, every node/.style={scale=0.5}]
        \node (1) at (0,0) [blackv] {};
        \node (2) at (2,0) [whitev] {};
        \node (12) at (1,1) [blackv] {};
        \node (123) at (2,2) [blackv] {};
        \node (23) at (3,1) [whitev] {};
        \node (3) at (4,0) [whitev] {};

      \end{tikzpicture},
      \begin{tikzpicture}[scale=0.2, every node/.style={scale=0.5}]
        \node (1) at (0,0) [blackv] {};
        \node (2) at (2,0) [whitev] {};
        \node (12) at (1,1) [whitev] {};
        \node (123) at (2,2) [blackv] {};
        \node (23) at (3,1) [whitev] {};
        \node (3) at (4,0) [blackv] {};

      \end{tikzpicture}
      \\ \hline
      \begin{tikzpicture}[scale=0.2, every node/.style={scale=0.5}]
        \node (1) at (0,0) [blackv] {};
        \node (2) at (2,0) [blackv] {};
        \node (12) at (1,1) [blackv] {};
        \node (123) at (2,2) [blackv] {};
        \node (23) at (3,1) [blackv] {};
        \node (3) at (4,0) [blackv] {};

      \end{tikzpicture}
       & 
      \begin{tikzpicture}[scale=0.2, every node/.style={scale=0.5}]
        \node (1) at (0,0) [blackv] {};
        \node (2) at (2,0) [whitev] {};
        \node (12) at (1,1) [blackv] {};
        \node (123) at (2,2) [blackv] {};
        \node (23) at (3,1) [whitev] {};
        \node (3) at (4,0) [whitev] {};

      \end{tikzpicture},
      \begin{tikzpicture}[scale=0.2, every node/.style={scale=0.5}]
        \node (1) at (0,0) [whitev] {};
        \node (2) at (2,0) [whitev] {};
        \node (12) at (1,1) [whitev] {};
        \node (123) at (2,2) [blackv] {};
        \node (23) at (3,1) [blackv] {};
        \node (3) at (4,0) [blackv] {};

      \end{tikzpicture}
      &
      \begin{tikzpicture}[scale=0.2, every node/.style={scale=0.5}]
        \node (1) at (0,0) [whitev] {};
        \node (2) at (2,0) [blackv] {};
        \node (12) at (1,1) [blackv] {};
        \node (123) at (2,2) [blackv] {};
        \node (23) at (3,1) [whitev] {};
        \node (3) at (4,0) [whitev] {};

      \end{tikzpicture}
       & 
      \begin{tikzpicture}[scale=0.2, every node/.style={scale=0.5}]
        \node (1) at (0,0) [whitev] {};
        \node (2) at (2,0) [blackv] {};
        \node (12) at (1,1) [blackv] {};
        \node (123) at (2,2) [blackv] {};
        \node (23) at (3,1) [whitev] {};
        \node (3) at (4,0) [whitev] {};
        
      \end{tikzpicture},
      \begin{tikzpicture}[scale=0.2, every node/.style={scale=0.5}]
        \node (1) at (0,0) [whitev] {};
        \node (2) at (2,0) [blackv] {};
        \node (12) at (1,1) [blackv] {};
        \node (123) at (2,2) [blackv] {};
        \node (23) at (3,1) [whitev] {};
        \node (3) at (4,0) [whitev] {};

      \end{tikzpicture}
      \\ \hline
      \begin{tikzpicture}[scale=0.2, every node/.style={scale=0.5}]
        \node (1) at (0,0) [whitev] {};
        \node (2) at (2,0) [blackv] {};
        \node (12) at (1,1) [blackv] {};
        \node (123) at (2,2) [blackv] {};
        \node (23) at (3,1) [blackv] {};
        \node (3) at (4,0) [whitev] {};

      \end{tikzpicture}
       & 
      \begin{tikzpicture}[scale=0.2, every node/.style={scale=0.5}]
        \node (1) at (0,0) [whitev] {};
        \node (2) at (2,0) [blackv] {};
        \node (12) at (1,1) [blackv] {};
        \node (123) at (2,2) [blackv] {};
        \node (23) at (3,1) [whitev] {};
        \node (3) at (4,0) [whitev] {};

      \end{tikzpicture},
      \begin{tikzpicture}[scale=0.2, every node/.style={scale=0.5}]
        \node (1) at (0,0) [whitev] {};
        \node (2) at (2,0) [blackv] {};
        \node (12) at (1,1) [whitev] {};
        \node (123) at (2,2) [blackv] {};
        \node (23) at (3,1) [blackv] {};
        \node (3) at (4,0) [whitev] {};

      \end{tikzpicture}
      &
      \begin{tikzpicture}[scale=0.2, every node/.style={scale=0.5}]
        \node (1) at (0,0) [whitev] {};
        \node (2) at (2,0) [blackv] {};
        \node (12) at (1,1) [blackv] {};
        \node (123) at (2,2) [blackv] {};
        \node (23) at (3,1) [blackv] {};
        \node (3) at (4,0) [blackv] {};

      \end{tikzpicture}
       & 
      \begin{tikzpicture}[scale=0.2, every node/.style={scale=0.5}]
        \node (1) at (0,0) [whitev] {};
        \node (2) at (2,0) [blackv] {};
        \node (12) at (1,1) [blackv] {};
        \node (123) at (2,2) [blackv] {};
        \node (23) at (3,1) [whitev] {};
        \node (3) at (4,0) [whitev] {};

      \end{tikzpicture},
      \begin{tikzpicture}[scale=0.2, every node/.style={scale=0.5}]
        \node (1) at (0,0) [whitev] {};
        \node (2) at (2,0) [whitev] {};
        \node (12) at (1,1) [whitev] {};
        \node (123) at (2,2) [blackv] {};
        \node (23) at (3,1) [blackv] {};
        \node (3) at (4,0) [blackv] {};

      \end{tikzpicture}
      \\ \hline
      \begin{tikzpicture}[scale=0.2, every node/.style={scale=0.5}]
        \node (1) at (0,0) [whitev] {};
        \node (2) at (2,0) [blackv] {};
        \node (12) at (1,1) [whitev] {};
        \node (123) at (2,2) [blackv] {};
        \node (23) at (3,1) [blackv] {};
        \node (3) at (4,0) [whitev] {};

      \end{tikzpicture}
       & 
      \begin{tikzpicture}[scale=0.2, every node/.style={scale=0.5}]
        \node (1) at (0,0) [whitev] {};
        \node (2) at (2,0) [blackv] {};
        \node (12) at (1,1) [whitev] {};
        \node (123) at (2,2) [blackv] {};
        \node (23) at (3,1) [blackv] {};
        \node (3) at (4,0) [whitev] {};

      \end{tikzpicture},
      \begin{tikzpicture}[scale=0.2, every node/.style={scale=0.5}]
        \node (1) at (0,0) [whitev] {};
        \node (2) at (2,0) [blackv] {};
        \node (12) at (1,1) [whitev] {};
        \node (123) at (2,2) [blackv] {};
        \node (23) at (3,1) [blackv] {};
        \node (3) at (4,0) [whitev] {};

      \end{tikzpicture}
      &
      \begin{tikzpicture}[scale=0.2, every node/.style={scale=0.5}]
        \node (1) at (0,0) [whitev] {};
        \node (2) at (2,0) [blackv] {};
        \node (12) at (1,1) [whitev] {};
        \node (123) at (2,2) [blackv] {};
        \node (23) at (3,1) [blackv] {};
        \node (3) at (4,0) [blackv] {};

      \end{tikzpicture}
       & 
      \begin{tikzpicture}[scale=0.2, every node/.style={scale=0.5}]
        \node (1) at (0,0) [whitev] {};
        \node (2) at (2,0) [blackv] {};
        \node (12) at (1,1) [whitev] {};
        \node (123) at (2,2) [blackv] {};
        \node (23) at (3,1) [blackv] {};
        \node (3) at (4,0) [whitev] {};

      \end{tikzpicture},
      \begin{tikzpicture}[scale=0.2, every node/.style={scale=0.5}]
        \node (1) at (0,0) [whitev] {};
        \node (2) at (2,0) [whitev] {};
        \node (12) at (1,1) [whitev] {};
        \node (123) at (2,2) [blackv] {};
        \node (23) at (3,1) [blackv] {};
        \node (3) at (4,0) [blackv] {};

      \end{tikzpicture}
      \\ \hline
      \begin{tikzpicture}[scale=0.2, every node/.style={scale=0.5}]
        \node (1) at (0,0) [blackv] {};
        \node (2) at (2,0) [whitev] {};
        \node (12) at (1,1) [whitev] {};
        \node (123) at (2,2) [blackv] {};
        \node (23) at (3,1) [whitev] {};
        \node (3) at (4,0) [blackv] {};

      \end{tikzpicture}
       & 
      \begin{tikzpicture}[scale=0.2, every node/.style={scale=0.5}]
        \node (1) at (0,0) [blackv] {};
        \node (2) at (2,0) [whitev] {};
        \node (12) at (1,1) [whitev] {};
        \node (123) at (2,2) [blackv] {};
        \node (23) at (3,1) [whitev] {};
        \node (3) at (4,0) [blackv] {};

      \end{tikzpicture},
      \begin{tikzpicture}[scale=0.2, every node/.style={scale=0.5}]
        \node (1) at (0,0) [blackv] {};
        \node (2) at (2,0) [whitev] {};
        \node (12) at (1,1) [whitev] {};
        \node (123) at (2,2) [blackv] {};
        \node (23) at (3,1) [whitev] {};
        \node (3) at (4,0) [blackv] {};

      \end{tikzpicture}
      &
      \begin{tikzpicture}[scale=0.2, every node/.style={scale=0.5}]
        \node (1) at (0,0) [blackv] {};
        \node (2) at (2,0) [whitev] {};
        \node (12) at (1,1) [whitev] {};
        \node (123) at (2,2) [blackv] {};
        \node (23) at (3,1) [blackv] {};
        \node (3) at (4,0) [blackv] {};

      \end{tikzpicture}
      & 
      \begin{tikzpicture}[scale=0.2, every node/.style={scale=0.5}]
        \node (1) at (0,0) [blackv] {};
        \node (2) at (2,0) [whitev] {};
        \node (12) at (1,1) [whitev] {};
        \node (123) at (2,2) [blackv] {};
        \node (23) at (3,1) [whitev] {};
        \node (3) at (4,0) [blackv] {};

      \end{tikzpicture},
      \begin{tikzpicture}[scale=0.2, every node/.style={scale=0.5}]
        \node (1) at (0,0) [whitev] {};
        \node (2) at (2,0) [whitev] {};
        \node (12) at (1,1) [whitev] {};
        \node (123) at (2,2) [blackv] {};
        \node (23) at (3,1) [blackv] {};
        \node (3) at (4,0) [blackv] {};

      \end{tikzpicture}
      \\ \hline
      \begin{tikzpicture}[scale=0.2, every node/.style={scale=0.5}]
        \node (1) at (0,0) [whitev] {};
        \node (2) at (2,0) [whitev] {};
        \node (12) at (1,1) [whitev] {};
        \node (123) at (2,2) [blackv] {};
        \node (23) at (3,1) [blackv] {};
        \node (3) at (4,0) [blackv] {};

      \end{tikzpicture}
       & 
      \begin{tikzpicture}[scale=0.2, every node/.style={scale=0.5}]
        \node (1) at (0,0) [whitev] {};
        \node (2) at (2,0) [whitev] {};
        \node (12) at (1,1) [whitev] {};
        \node (123) at (2,2) [blackv] {};
        \node (23) at (3,1) [blackv] {};
        \node (3) at (4,0) [blackv] {};

      \end{tikzpicture},
      \begin{tikzpicture}[scale=0.2, every node/.style={scale=0.5}]
        \node (1) at (0,0) [whitev] {};
        \node (2) at (2,0) [whitev] {};
        \node (12) at (1,1) [whitev] {};
        \node (123) at (2,2) [blackv] {};
        \node (23) at (3,1) [blackv] {};
        \node (3) at (4,0) [blackv] {};

      \end{tikzpicture}
      &
      \begin{tikzpicture}[scale=0.2, every node/.style={scale=0.5}]
        \node (1) at (0,0) [blackv] {};
        \node (2) at (2,0) [whitev] {};
        \node (12) at (1,1) [blackv] {};
        \node (123) at (2,2) [whitev] {};
        \node (23) at (3,1) [whitev] {};
        \node (3) at (4,0) [whitev] {};

      \end{tikzpicture}
       & 
      \begin{tikzpicture}[scale=0.2, every node/.style={scale=0.5}]
        \node (1) at (0,0) [blackv] {};
        \node (2) at (2,0) [whitev] {};
        \node (12) at (1,1) [blackv] {};
        \node (123) at (2,2) [whitev] {};
        \node (23) at (3,1) [whitev] {};
        \node (3) at (4,0) [whitev] {};

      \end{tikzpicture},
      \begin{tikzpicture}[scale=0.2, every node/.style={scale=0.5}]
        \node (1) at (0,0) [blackv] {};
        \node (2) at (2,0) [whitev] {};
        \node (12) at (1,1) [blackv] {};
        \node (123) at (2,2) [whitev] {};
        \node (23) at (3,1) [whitev] {};
        \node (3) at (4,0) [whitev] {};

      \end{tikzpicture}
      \\ \hline
      \begin{tikzpicture}[scale=0.2, every node/.style={scale=0.5}]
        \node (1) at (0,0) [blackv] {};
        \node (2) at (2,0) [blackv] {};
        \node (12) at (1,1) [blackv] {};
        \node (123) at (2,2) [whitev] {};
        \node (23) at (3,1) [whitev] {};
        \node (3) at (4,0) [whitev] {};

      \end{tikzpicture}
       & 
      \begin{tikzpicture}[scale=0.2, every node/.style={scale=0.5}]
        \node (1) at (0,0) [blackv] {};
        \node (2) at (2,0) [whitev] {};
        \node (12) at (1,1) [blackv] {};
        \node (123) at (2,2) [whitev] {};
        \node (23) at (3,1) [whitev] {};
        \node (3) at (4,0) [whitev] {};

      \end{tikzpicture},
      \begin{tikzpicture}[scale=0.2, every node/.style={scale=0.5}]
        \node (1) at (0,0) [whitev] {};
        \node (2) at (2,0) [blackv] {};
        \node (12) at (1,1) [blackv] {};
        \node (123) at (2,2) [whitev] {};
        \node (23) at (3,1) [whitev] {};
        \node (3) at (4,0) [whitev] {};

      \end{tikzpicture}
      &
      \begin{tikzpicture}[scale=0.2, every node/.style={scale=0.5}]
        \node (1) at (0,0) [blackv] {};
        \node (2) at (2,0) [whitev] {};
        \node (12) at (1,1) [whitev] {};
        \node (123) at (2,2) [blackv] {};
        \node (23) at (3,1) [whitev] {};
        \node (3) at (4,0) [whitev] {};

      \end{tikzpicture}
       & 
      \begin{tikzpicture}[scale=0.2, every node/.style={scale=0.5}]
        \node (1) at (0,0) [blackv] {};
        \node (2) at (2,0) [whitev] {};
        \node (12) at (1,1) [whitev] {};
        \node (123) at (2,2) [blackv] {};
        \node (23) at (3,1) [whitev] {};
        \node (3) at (4,0) [whitev] {};

      \end{tikzpicture},
      \begin{tikzpicture}[scale=0.2, every node/.style={scale=0.5}]
        \node (1) at (0,0) [blackv] {};
        \node (2) at (2,0) [whitev] {};
        \node (12) at (1,1) [whitev] {};
        \node (123) at (2,2) [blackv] {};
        \node (23) at (3,1) [whitev] {};
        \node (3) at (4,0) [whitev] {};

      \end{tikzpicture}
      \\ \hline
      \begin{tikzpicture}[scale=0.2, every node/.style={scale=0.5}]
        \node (1) at (0,0) [blackv] {};
        \node (2) at (2,0) [whitev] {};
        \node (12) at (1,1) [whitev] {};
        \node (123) at (2,2) [blackv] {};
        \node (23) at (3,1) [blackv] {};
        \node (3) at (4,0) [whitev] {};

      \end{tikzpicture}
       & 
      \begin{tikzpicture}[scale=0.2, every node/.style={scale=0.5}]
        \node (1) at (0,0) [blackv] {};
        \node (2) at (2,0) [whitev] {};
        \node (12) at (1,1) [whitev] {};
        \node (123) at (2,2) [blackv] {};
        \node (23) at (3,1) [whitev] {};
        \node (3) at (4,0) [whitev] {};

      \end{tikzpicture},
      \begin{tikzpicture}[scale=0.2, every node/.style={scale=0.5}]
        \node (1) at (0,0) [whitev] {};
        \node (2) at (2,0) [whitev] {};
        \node (12) at (1,1) [whitev] {};
        \node (123) at (2,2) [blackv] {};
        \node (23) at (3,1) [blackv] {};
        \node (3) at (4,0) [whitev] {};

      \end{tikzpicture}
      &
      \begin{tikzpicture}[scale=0.2, every node/.style={scale=0.5}]
        \node (1) at (0,0) [blackv] {};
        \node (2) at (2,0) [whitev] {};
        \node (12) at (1,1) [whitev] {};
        \node (123) at (2,2) [whitev] {};
        \node (23) at (3,1) [whitev] {};
        \node (3) at (4,0) [blackv] {};

      \end{tikzpicture}
      & 
      \begin{tikzpicture}[scale=0.2, every node/.style={scale=0.5}]
        \node (1) at (0,0) [blackv] {};
        \node (2) at (2,0) [whitev] {};
        \node (12) at (1,1) [whitev] {};
        \node (123) at (2,2) [whitev] {};
        \node (23) at (3,1) [whitev] {};
        \node (3) at (4,0) [blackv] {};

      \end{tikzpicture},
      \begin{tikzpicture}[scale=0.2, every node/.style={scale=0.5}]
        \node (1) at (0,0) [blackv] {};
        \node (2) at (2,0) [whitev] {};
        \node (12) at (1,1) [whitev] {};
        \node (123) at (2,2) [whitev] {};
        \node (23) at (3,1) [whitev] {};
        \node (3) at (4,0) [blackv] {};

      \end{tikzpicture}
      \\ \hline
      \begin{tikzpicture}[scale=0.2, every node/.style={scale=0.5}]
        \node (1) at (0,0) [whitev] {};
        \node (2) at (2,0) [whitev] {};
        \node (12) at (1,1) [blackv] {};
        \node (123) at (2,2) [blackv] {};
        \node (23) at (3,1) [whitev] {};
        \node (3) at (4,0) [whitev] {};

      \end{tikzpicture}
       & 
      \begin{tikzpicture}[scale=0.2, every node/.style={scale=0.5}]
        \node (1) at (0,0) [whitev] {};
        \node (2) at (2,0) [whitev] {};
        \node (12) at (1,1) [blackv] {};
        \node (123) at (2,2) [blackv] {};
        \node (23) at (3,1) [whitev] {};
        \node (3) at (4,0) [whitev] {};

      \end{tikzpicture},
      \begin{tikzpicture}[scale=0.2, every node/.style={scale=0.5}]
        \node (1) at (0,0) [whitev] {};
        \node (2) at (2,0) [whitev] {};
        \node (12) at (1,1) [blackv] {};
        \node (123) at (2,2) [blackv] {};
        \node (23) at (3,1) [whitev] {};
        \node (3) at (4,0) [whitev] {};

      \end{tikzpicture}
      &
      \begin{tikzpicture}[scale=0.2, every node/.style={scale=0.5}]
        \node (1) at (0,0) [whitev] {};
        \node (2) at (2,0) [whitev] {};
        \node (12) at (1,1) [blackv] {};
        \node (123) at (2,2) [blackv] {};
        \node (23) at (3,1) [whitev] {};
        \node (3) at (4,0) [blackv] {};

      \end{tikzpicture}
       & 
      \begin{tikzpicture}[scale=0.2, every node/.style={scale=0.5}]
        \node (1) at (0,0) [whitev] {};
        \node (2) at (2,0) [whitev] {};
        \node (12) at (1,1) [blackv] {};
        \node (123) at (2,2) [blackv] {};
        \node (23) at (3,1) [whitev] {};
        \node (3) at (4,0) [whitev] {};

      \end{tikzpicture},
      \begin{tikzpicture}[scale=0.2, every node/.style={scale=0.5}]
        \node (1) at (0,0) [whitev] {};
        \node (2) at (2,0) [whitev] {};
        \node (12) at (1,1) [whitev] {};
        \node (123) at (2,2) [blackv] {};
        \node (23) at (3,1) [whitev] {};
        \node (3) at (4,0) [blackv] {};

      \end{tikzpicture}
      \\ \hline
      \begin{tikzpicture}[scale=0.2, every node/.style={scale=0.5}]
        \node (1) at (0,0) [whitev] {};
        \node (2) at (2,0) [blackv] {};
        \node (12) at (1,1) [blackv] {};
        \node (123) at (2,2) [whitev] {};
        \node (23) at (3,1) [whitev] {};
        \node (3) at (4,0) [whitev] {};

      \end{tikzpicture}
       & 
      \begin{tikzpicture}[scale=0.2, every node/.style={scale=0.5}]
        \node (1) at (0,0) [whitev] {};
        \node (2) at (2,0) [blackv] {};
        \node (12) at (1,1) [blackv] {};
        \node (123) at (2,2) [whitev] {};
        \node (23) at (3,1) [whitev] {};
        \node (3) at (4,0) [whitev] {};

      \end{tikzpicture},
      \begin{tikzpicture}[scale=0.2, every node/.style={scale=0.5}]
        \node (1) at (0,0) [whitev] {};
        \node (2) at (2,0) [blackv] {};
        \node (12) at (1,1) [blackv] {};
        \node (123) at (2,2) [whitev] {};
        \node (23) at (3,1) [whitev] {};
        \node (3) at (4,0) [whitev] {};

      \end{tikzpicture}
      &
      \begin{tikzpicture}[scale=0.2, every node/.style={scale=0.5}]
        \node (1) at (0,0) [whitev] {};
        \node (2) at (2,0) [blackv] {};
        \node (12) at (1,1) [whitev] {};
        \node (123) at (2,2) [blackv] {};
        \node (23) at (3,1) [whitev] {};
        \node (3) at (4,0) [whitev] {};

      \end{tikzpicture}
       & 
      \begin{tikzpicture}[scale=0.2, every node/.style={scale=0.5}]
        \node (1) at (0,0) [whitev] {};
        \node (2) at (2,0) [blackv] {};
        \node (12) at (1,1) [whitev] {};
        \node (123) at (2,2) [blackv] {};
        \node (23) at (3,1) [whitev] {};
        \node (3) at (4,0) [whitev] {};

      \end{tikzpicture},
      \begin{tikzpicture}[scale=0.2, every node/.style={scale=0.5}]
        \node (1) at (0,0) [whitev] {};
        \node (2) at (2,0) [blackv] {};
        \node (12) at (1,1) [whitev] {};
        \node (123) at (2,2) [blackv] {};
        \node (23) at (3,1) [whitev] {};
        \node (3) at (4,0) [whitev] {};

      \end{tikzpicture}
      \\ \hline
      \begin{tikzpicture}[scale=0.2, every node/.style={scale=0.5}]
        \node (1) at (0,0) [whitev] {};
        \node (2) at (2,0) [whitev] {};
        \node (12) at (1,1) [whitev] {};
        \node (123) at (2,2) [blackv] {};
        \node (23) at (3,1) [blackv] {};
        \node (3) at (4,0) [whitev] {};

      \end{tikzpicture}
       & 
      \begin{tikzpicture}[scale=0.2, every node/.style={scale=0.5}]
        \node (1) at (0,0) [whitev] {};
        \node (2) at (2,0) [whitev] {};
        \node (12) at (1,1) [whitev] {};
        \node (123) at (2,2) [blackv] {};
        \node (23) at (3,1) [blackv] {};
        \node (3) at (4,0) [whitev] {};

      \end{tikzpicture},
      \begin{tikzpicture}[scale=0.2, every node/.style={scale=0.5}]
        \node (1) at (0,0) [whitev] {};
        \node (2) at (2,0) [whitev] {};
        \node (12) at (1,1) [whitev] {};
        \node (123) at (2,2) [blackv] {};
        \node (23) at (3,1) [blackv] {};
        \node (3) at (4,0) [whitev] {};

      \end{tikzpicture}
      &
      \begin{tikzpicture}[scale=0.2, every node/.style={scale=0.5}]
        \node (1) at (0,0) [whitev] {};
        \node (2) at (2,0) [whitev] {};
        \node (12) at (1,1) [whitev] {};
        \node (123) at (2,2) [blackv] {};
        \node (23) at (3,1) [whitev] {};
        \node (3) at (4,0) [blackv] {};

      \end{tikzpicture}
       & 
      \begin{tikzpicture}[scale=0.2, every node/.style={scale=0.5}]
        \node (1) at (0,0) [whitev] {};
        \node (2) at (2,0) [whitev] {};
        \node (12) at (1,1) [whitev] {};
        \node (123) at (2,2) [blackv] {};
        \node (23) at (3,1) [whitev] {};
        \node (3) at (4,0) [blackv] {};

      \end{tikzpicture},
      \begin{tikzpicture}[scale=0.2, every node/.style={scale=0.5}]
        \node (1) at (0,0) [whitev] {};
        \node (2) at (2,0) [whitev] {};
        \node (12) at (1,1) [whitev] {};
        \node (123) at (2,2) [blackv] {};
        \node (23) at (3,1) [whitev] {};
        \node (3) at (4,0) [blackv] {};

      \end{tikzpicture}
      \\ \hline
      \begin{tikzpicture}[scale=0.2, every node/.style={scale=0.5}]
        \node (1) at (0,0) [whitev] {};
        \node (2) at (2,0) [blackv] {};
        \node (12) at (1,1) [whitev] {};
        \node (123) at (2,2) [whitev] {};
        \node (23) at (3,1) [blackv] {};
        \node (3) at (4,0) [whitev] {};

      \end{tikzpicture}
       & 
      \begin{tikzpicture}[scale=0.2, every node/.style={scale=0.5}]
        \node (1) at (0,0) [whitev] {};
        \node (2) at (2,0) [blackv] {};
        \node (12) at (1,1) [whitev] {};
        \node (123) at (2,2) [whitev] {};
        \node (23) at (3,1) [blackv] {};
        \node (3) at (4,0) [whitev] {};

      \end{tikzpicture},
      \begin{tikzpicture}[scale=0.2, every node/.style={scale=0.5}]
        \node (1) at (0,0) [whitev] {};
        \node (2) at (2,0) [blackv] {};
        \node (12) at (1,1) [whitev] {};
        \node (123) at (2,2) [whitev] {};
        \node (23) at (3,1) [blackv] {};
        \node (3) at (4,0) [whitev] {};

      \end{tikzpicture}
      &
      \begin{tikzpicture}[scale=0.2, every node/.style={scale=0.5}]
        \node (1) at (0,0) [whitev] {};
        \node (2) at (2,0) [blackv] {};
        \node (12) at (1,1) [whitev] {};
        \node (123) at (2,2) [whitev] {};
        \node (23) at (3,1) [blackv] {};
        \node (3) at (4,0) [blackv] {};

      \end{tikzpicture}
       & 
      \begin{tikzpicture}[scale=0.2, every node/.style={scale=0.5}]
        \node (1) at (0,0) [whitev] {};
        \node (2) at (2,0) [blackv] {};
        \node (12) at (1,1) [whitev] {};
        \node (123) at (2,2) [whitev] {};
        \node (23) at (3,1) [blackv] {};
        \node (3) at (4,0) [whitev] {};

      \end{tikzpicture},
      \begin{tikzpicture}[scale=0.2, every node/.style={scale=0.5}]
        \node (1) at (0,0) [whitev] {};
        \node (2) at (2,0) [whitev] {};
        \node (12) at (1,1) [whitev] {};
        \node (123) at (2,2) [whitev] {};
        \node (23) at (3,1) [blackv] {};
        \node (3) at (4,0) [blackv] {};

      \end{tikzpicture}
      \\ \hline
      \begin{tikzpicture}[scale=0.2, every node/.style={scale=0.5}]
        \node (1) at (0,0) [whitev] {};
        \node (2) at (2,0) [whitev] {};
        \node (12) at (1,1) [whitev] {};
        \node (123) at (2,2) [whitev] {};
        \node (23) at (3,1) [blackv] {};
        \node (3) at (4,0) [blackv] {};

      \end{tikzpicture}
       & 
      \begin{tikzpicture}[scale=0.2, every node/.style={scale=0.5}]
        \node (1) at (0,0) [whitev] {};
        \node (2) at (2,0) [whitev] {};
        \node (12) at (1,1) [whitev] {};
        \node (123) at (2,2) [whitev] {};
        \node (23) at (3,1) [blackv] {};
        \node (3) at (4,0) [blackv] {};

      \end{tikzpicture},
      \begin{tikzpicture}[scale=0.2, every node/.style={scale=0.5}]
        \node (1) at (0,0) [whitev] {};
        \node (2) at (2,0) [whitev] {};
        \node (12) at (1,1) [whitev] {};
        \node (123) at (2,2) [whitev] {};
        \node (23) at (3,1) [blackv] {};
        \node (3) at (4,0) [blackv] {};

      \end{tikzpicture}
      &
      \begin{tikzpicture}[scale=0.2, every node/.style={scale=0.5}]
        \node (1) at (0,0) [blackv] {};
        \node (2) at (2,0) [whitev] {};
        \node (12) at (1,1) [whitev] {};
        \node (123) at (2,2) [whitev] {};
        \node (23) at (3,1) [whitev] {};
        \node (3) at (4,0) [whitev] {};

      \end{tikzpicture}
       & 
      \begin{tikzpicture}[scale=0.2, every node/.style={scale=0.5}]
        \node (1) at (0,0) [blackv] {};
        \node (2) at (2,0) [whitev] {};
        \node (12) at (1,1) [whitev] {};
        \node (123) at (2,2) [whitev] {};
        \node (23) at (3,1) [whitev] {};
        \node (3) at (4,0) [whitev] {};

      \end{tikzpicture},
      \begin{tikzpicture}[scale=0.2, every node/.style={scale=0.5}]
        \node (1) at (0,0) [blackv] {};
        \node (2) at (2,0) [whitev] {};
        \node (12) at (1,1) [whitev] {};
        \node (123) at (2,2) [whitev] {};
        \node (23) at (3,1) [whitev] {};
        \node (3) at (4,0) [whitev] {};

      \end{tikzpicture}
      \\ \hline
      \begin{tikzpicture}[scale=0.2, every node/.style={scale=0.5}]
        \node (1) at (0,0) [whitev] {};
        \node (2) at (2,0) [whitev] {};
        \node (12) at (1,1) [blackv] {};
        \node (123) at (2,2) [whitev] {};
        \node (23) at (3,1) [whitev] {};
        \node (3) at (4,0) [whitev] {};

      \end{tikzpicture}
       & 
      \begin{tikzpicture}[scale=0.2, every node/.style={scale=0.5}]
        \node (1) at (0,0) [whitev] {};
        \node (2) at (2,0) [whitev] {};
        \node (12) at (1,1) [blackv] {};
        \node (123) at (2,2) [whitev] {};
        \node (23) at (3,1) [whitev] {};
        \node (3) at (4,0) [whitev] {};

      \end{tikzpicture},
      \begin{tikzpicture}[scale=0.2, every node/.style={scale=0.5}]
        \node (1) at (0,0) [whitev] {};
        \node (2) at (2,0) [whitev] {};
        \node (12) at (1,1) [blackv] {};
        \node (123) at (2,2) [whitev] {};
        \node (23) at (3,1) [whitev] {};
        \node (3) at (4,0) [whitev] {};

      \end{tikzpicture}
      &
      \begin{tikzpicture}[scale=0.2, every node/.style={scale=0.5}]
        \node (1) at (0,0) [whitev] {};
        \node (2) at (2,0) [whitev] {};
        \node (12) at (1,1) [whitev] {};
        \node (123) at (2,2) [blackv] {};
        \node (23) at (3,1) [whitev] {};
        \node (3) at (4,0) [whitev] {};

      \end{tikzpicture}
       & 
      \begin{tikzpicture}[scale=0.2, every node/.style={scale=0.5}]
        \node (1) at (0,0) [whitev] {};
        \node (2) at (2,0) [whitev] {};
        \node (12) at (1,1) [whitev] {};
        \node (123) at (2,2) [blackv] {};
        \node (23) at (3,1) [whitev] {};
        \node (3) at (4,0) [whitev] {};

      \end{tikzpicture},
      \begin{tikzpicture}[scale=0.2, every node/.style={scale=0.5}]
        \node (1) at (0,0) [whitev] {};
        \node (2) at (2,0) [whitev] {};
        \node (12) at (1,1) [whitev] {};
        \node (123) at (2,2) [blackv] {};
        \node (23) at (3,1) [whitev] {};
        \node (3) at (4,0) [whitev] {};

      \end{tikzpicture}
      \\ \hline
      \begin{tikzpicture}[scale=0.2, every node/.style={scale=0.5}]
        \node (1) at (0,0) [whitev] {};
        \node (2) at (2,0) [blackv] {};
        \node (12) at (1,1) [whitev] {};
        \node (123) at (2,2) [whitev] {};
        \node (23) at (3,1) [whitev] {};
        \node (3) at (4,0) [whitev] {};

      \end{tikzpicture}
       & 
      \begin{tikzpicture}[scale=0.2, every node/.style={scale=0.5}]
        \node (1) at (0,0) [whitev] {};
        \node (2) at (2,0) [blackv] {};
        \node (12) at (1,1) [whitev] {};
        \node (123) at (2,2) [whitev] {};
        \node (23) at (3,1) [whitev] {};
        \node (3) at (4,0) [whitev] {};

      \end{tikzpicture},
      \begin{tikzpicture}[scale=0.2, every node/.style={scale=0.5}]
        \node (1) at (0,0) [whitev] {};
        \node (2) at (2,0) [blackv] {};
        \node (12) at (1,1) [whitev] {};
        \node (123) at (2,2) [whitev] {};
        \node (23) at (3,1) [whitev] {};
        \node (3) at (4,0) [whitev] {};

      \end{tikzpicture}
      &
      \begin{tikzpicture}[scale=0.2, every node/.style={scale=0.5}]
        \node (1) at (0,0) [whitev] {};
        \node (2) at (2,0) [whitev] {};
        \node (12) at (1,1) [whitev] {};
        \node (123) at (2,2) [whitev] {};
        \node (23) at (3,1) [blackv] {};
        \node (3) at (4,0) [whitev] {};

      \end{tikzpicture}
       & 
      \begin{tikzpicture}[scale=0.2, every node/.style={scale=0.5}]
        \node (1) at (0,0) [whitev] {};
        \node (2) at (2,0) [whitev] {};
        \node (12) at (1,1) [whitev] {};
        \node (123) at (2,2) [whitev] {};
        \node (23) at (3,1) [blackv] {};
        \node (3) at (4,0) [whitev] {};

      \end{tikzpicture},
      \begin{tikzpicture}[scale=0.2, every node/.style={scale=0.5}]
        \node (1) at (0,0) [whitev] {};
        \node (2) at (2,0) [whitev] {};
        \node (12) at (1,1) [whitev] {};
        \node (123) at (2,2) [whitev] {};
        \node (23) at (3,1) [blackv] {};
        \node (3) at (4,0) [whitev] {};

      \end{tikzpicture}
      \\ \hline
      \begin{tikzpicture}[scale=0.2, every node/.style={scale=0.5}]
        \node (1) at (0,0) [whitev] {};
        \node (2) at (2,0) [whitev] {};
        \node (12) at (1,1) [whitev] {};
        \node (123) at (2,2) [whitev] {};
        \node (23) at (3,1) [whitev] {};
        \node (3) at (4,0) [blackv] {};

      \end{tikzpicture}
       & 
      \begin{tikzpicture}[scale=0.2, every node/.style={scale=0.5}]
        \node (1) at (0,0) [whitev] {};
        \node (2) at (2,0) [whitev] {};
        \node (12) at (1,1) [whitev] {};
        \node (123) at (2,2) [whitev] {};
        \node (23) at (3,1) [whitev] {};
        \node (3) at (4,0) [blackv] {};

      \end{tikzpicture},
      \begin{tikzpicture}[scale=0.2, every node/.style={scale=0.5}]
        \node (1) at (0,0) [whitev] {};
        \node (2) at (2,0) [whitev] {};
        \node (12) at (1,1) [whitev] {};
        \node (123) at (2,2) [whitev] {};
        \node (23) at (3,1) [whitev] {};
        \node (3) at (4,0) [blackv] {};

      \end{tikzpicture}
      &
      \begin{tikzpicture}[scale=0.2, every node/.style={scale=0.5}]
        \node (1) at (0,0) [whitev] {};
        \node (2) at (2,0) [whitev] {};
        \node (12) at (1,1) [whitev] {};
        \node (123) at (2,2) [whitev] {};
        \node (23) at (3,1) [whitev] {};
        \node (3) at (4,0) [whitev] {};

      \end{tikzpicture}
       & 
      \begin{tikzpicture}[scale=0.2, every node/.style={scale=0.5}]
        \node (1) at (0,0) [whitev] {};
        \node (2) at (2,0) [whitev] {};
        \node (12) at (1,1) [whitev] {};
        \node (123) at (2,2) [whitev] {};
        \node (23) at (3,1) [whitev] {};
        \node (3) at (4,0) [whitev] {};

      \end{tikzpicture},
      \begin{tikzpicture}[scale=0.2, every node/.style={scale=0.5}]
        \node (1) at (0,0) [whitev] {};
        \node (2) at (2,0) [whitev] {};
        \node (12) at (1,1) [whitev] {};
        \node (123) at (2,2) [whitev] {};
        \node (23) at (3,1) [whitev] {};
        \node (3) at (4,0) [whitev] {};

      \end{tikzpicture}
      \\ \hline
      \caption{IE-closed subcategories $\CC$ and twin rigid modules $(P,I)$}
      \label{tab:ie}
      \end{longtable}
\end{example}

A subcategory of an abelian category which is closed under images, cokernels, and extensions are called \emph{ICE-closed}, and it has been recently studied by the authors (\cite{Eno, ES}).
While the study of IE-closed subcategories is partially motivated by that of ICE-closed subcategories, IE-closed subcategories behave differently from ICE-closed subcategories as follows.
\begin{remark}
For a Dynkin quiver $Q$, the number of ICE-closed subcategories of $\mod kQ$ depends only on the underlying graph of $Q$, not on the orientation of $Q$ (\cite[Theorem C]{Eno}).
However, the number of IE-closed subcategories depends on the orientation.
Indeed, there are 34 IE-closed subcategories in the previous example, while there are 35 for $Q\colon 1\to 2\ot 3$.
\end{remark}

\begin{remark}
In \cite{Eno2}, the first author studied the notion of a \emph{torsion heart}, which is a subcategory of $\mod\Lambda$ of the form $\UU^{\perp}\cap\TT$ for two torsion classes $\UU\subseteq\TT$ in $\mod\Lambda$ (here $\UU^{\perp} = \{X \in \mod\Lambda \mid \Hom_\Lambda(\UU, X)=0 \}$).
Since $\UU^\perp$ is a torsion-free class, every torsion heart is IE-closed.
However, the converse does not hold in general.
For example, we can check that an IE-closed subcategory $\add(\sst{1}\oplus\sst{3 \\ 2 \\ 1}\oplus\sst{3})$ in Example \ref{ex:ex} is not a torsion heart.
This is in contrast to ICE-closed subcategories, because every ICE-closed subcategories of $\mod\Lambda$ is a torsion heart by \cite[Proposition 3.1]{ES}.
\end{remark}

\appendix

\section{Tilting exchange sequences}\label{sec:app}
In this section, we collect basic facts about tilting modules, which are mainly used in Section~\ref{sec:comandmut}. Throughout this section, $\Lambda$ denotes an artin algebra.

\begin{definition}
  Let $T \in \mod\Lambda$.
  \begin{enumerate}
    \item $T$ is \emph{rigid} if $\Ext_\Lambda^1(T, T) = 0$ holds.
    \item $T$ is \emph{partial tilting} if $T$ is rigid and $\pd_\Lambda T \leq 1$, that is, there is a short exact sequence
    \[
      \begin{tikzcd}
        0 \rar & P_1 \rar & P_0 \rar & T \rar & 0
      \end{tikzcd}
    \]
    in $\mod\Lambda$ with $P_0,P_1 \in \proj \Lambda$.
    \item $T$ is \emph{tilting} if $T$ is partial tilting and there is a short exact sequence
    \[
      \begin{tikzcd}
        0 \rar & \Lambda \rar & T^0 \rar & T^1 \rar & 0
      \end{tikzcd}
    \]
    in $\mod\Lambda$ with $T^0, T^1 \in \add T$.
  \end{enumerate}
\end{definition}

\begin{proposition}\label{prop:tilt-comp-BB}
  Let $M \in \mod\Lambda$. Then the following statements hold.
  \begin{enumerate}
    \item Suppose that $M$ is partial tilting. Then $|M| \leq |\Lambda|$ holds, and $M$ is tilting if and only if $|M| = |\Lambda|$ holds. Also, there is some $E \in \mod\Lambda$ such that $E \oplus M$ is tilting (called the Bongartz completion) and that there is a short exact sequence
    \[
      \begin{tikzcd}
        0 \rar & \Lambda \rar & E \rar & M_0 \rar & 0
      \end{tikzcd}
    \]
    with $M_0 \in \add M$.
    \item Suppose that $M$ is tilting. Then $\Fac M$ has an $\Ext$-progenerator $M$. Moreover, put $\Gamma := \End_\Lambda M$. Then $_\Gamma M$ is a tilting left $\Gamma$-module, and $\Hom_\Lambda(M,-) \colon \mod\Lambda \to \mod\Gamma$ induces an exact equivalence $\Fac M \equi \Sub (DM)$ for a torsion-free class $\Sub (DM)$ in $\mod\Gamma$.
  \end{enumerate}
\end{proposition}

\begin{definition}
  A $\Lambda$-module $M \in \mod\Lambda$ is called \emph{almost complete tilting} if $M$ is partial tilting and $|M| = |\Lambda| - 1$ holds. In this case, an \emph{indecomposable} module $X$ such that $X \oplus M$ is tilting is called a \emph{tilting complement of $M$}.
\end{definition}

An almost complete tilting module $M$ always has a complement $X$ such that $X \oplus M$ is the Bongartz completion, and we call it the \emph{Bongartz complement}.
It is known that the Bongartz complement and the other complement (if exists) is related by the following special short exact sequence.
\begin{definition}
  A \emph{tilting exchange sequence} is a short exact sequence
  \begin{equation}\label{eq:tilt-exchange-seq}
    \begin{tikzcd}
      0 \rar & X \rar["f"] & \widetilde{M} \rar["g"] & Y \rar & 0
    \end{tikzcd}
  \end{equation}
  satisfying the following conditions:
  \begin{enumerate}
    \item $\widetilde{M} \in \add M$ for an almost complete tilting module $M$.
    \item $X$ and $Y$ are two tilting complements of $M$ (in particular, they are indecomposable).
  \end{enumerate}
  In this case, $f$ (resp. $g$) is automatically a left (resp. right) minimal $\add M$-approximation, and we call \eqref{eq:tilt-exchange-seq} a \emph{tilting exchange sequence from $X \oplus M$ to $Y \oplus M$}. 
\end{definition}

\begin{theorem}[{\cite[Proposition 1.3]{RS}}]\label{thm:exists-exchange-seq}
  Let $M$ be an almost complete tilting $\Lambda$-module. Then the following holds.
  \begin{enumerate}
    \item There are at most two non-isomorphic tilting complements of $M$.
    \item Suppose that there are two non-isomorphic tilting complements $X$ and $Y$. Then $X$ is the Bongartz complement of $M$ if and only if $Y \in \Fac M$ if and only if there is a tilting exchange sequence \eqref{eq:tilt-exchange-seq} from $X \oplus M$ to $Y \oplus M$.
  \end{enumerate}
\end{theorem}

On the other hand, there is a natural poset structure on the set of tilting modules, which is induced from torsion classes.
\begin{definition}\label{def:tilt}
  Let $\tilt\Lambda$ denote the set of isomorphism classes of basic tilting $\Lambda$-modules.
  Define a partial order on $\tilt\Lambda$ by $T \geq T'$ if $\Fac T \supseteq \Fac T'$.
  Denote by $\HH(\tilt\Lambda)$ the \emph{Hasse quiver} of $\tilt\Lambda$, that is, we draw an arrow $T \to T'$ if $T > T'$ and there is no $T'' \in \tilt\Lambda$ satisfying $T > T'' > T'$.
\end{definition}
Then the following relation between tilting exchange sequences and the quiver $\HH(\tilt\Lambda)$ is known.
\begin{theorem}\label{thm:hasse}
  The following are equivalent for $T, T' \in \tilt\Lambda$.
  \begin{enumerate}
    \item There is an arrow $T \to T'$ in $\HH(\tilt\Lambda)$.
    \item There is a tilting exchange sequence from $T$ to $T'$.
  \end{enumerate}
\end{theorem}
\begin{proof}
  This is well-known to experts, but we have not been able to find a precise reference, so we give a brief explanation using $\tau$-tilting theory \cite{AIR}.
  If $T \geq T'$ holds for two support $\tau$-tilting modules and $T'$ is tilting, then so is $T$. Therefore $\HH(\tilt\Lambda)$ is a full subquiver of the Hasse quiver of support $\tau$-tilting modules, and by \cite[Theorem 2.33]{AIR}, the condition (1) is equivalent to that $T = M \oplus X$ and $T' = M \oplus Y$ for some almost complete tilting module $M$ with $T$ being the Bongartz completion of $M$.
  Then Theorem \ref{thm:exists-exchange-seq} implies that this is equivalent to (2).
\end{proof}
\begin{remark}
  This can be also deduced from the results of Happel--Unger about generalized tilting modules. In \cite[Theorem 2.1]{HU}, the result analogous of the above theorem for generalized tilting modules is shown. On the other hand, if $T \geq T'$ holds for two generalized tilting modules and $T'$ is tilting, then so is $T$ by \cite[Lemma 2.1]{HU2}. Thus $\HH(\tilt\Lambda)$ is a full subquiver of the Hasse quiver of generalized tilting modules. Hence the above theorem follows.
\end{remark}

The following gives a criterion when there is a tilting exchange sequence from $X \oplus M$ to some $Y \oplus M$. This also seems to be well-known to experts, but we include a proof for the convenience of the reader.
\begin{proposition}\label{prop:tilt-mutation}
  Let $M$ be an almost complete tilting $\Lambda$-module and $X$ a tilting complement of $M$. Then the following conditions are equivalent:
  \begin{enumerate}
    \item There is some tilting complement $Y$ of $M$ together with a tilting exchange sequence from $X \oplus M$ to $Y \oplus M$.
    \item $X \in \Sub M$ holds, and $\pd_\Lambda \coker f \leq 1$ holds for a left minimal $\add M$-approximation $f \colon X \to \widetilde{M}$. 
  \end{enumerate}
  Moreover, in {\upshape (2)}, the following sequence
  \begin{equation}\label{eq:tilt-mut-seq}
    \begin{tikzcd}
      0 \rar & X \rar["f"] & \widetilde{M} \rar & \coker f \rar & 0
    \end{tikzcd}
  \end{equation}
  is a tilting exchange sequence, that is, $\coker f$ is indecomposable and $M \oplus \coker f$ is tilting.
\end{proposition}
\begin{proof}
  (1) $\Rightarrow$ (2):
  This is clear since the left map of the tilting exchange sequence from $M \oplus X$ to $M \oplus Y$ is a left minimal $\add M$-approximation.

  (2) $\Rightarrow$ (1):
  Let $f \colon X \to \widetilde{M}$ be a left minimal $\add M$-approximation, which is injective by $X \in \Sub M$. Since we assume $\pd_\Lambda \coker f \leq 1$, it is easily checked from \eqref{eq:tilt-mut-seq} that $M \oplus \coker f$ is partial tilting. Also we have $\coker f \not\in \add M$, since otherwise \eqref{eq:tilt-mut-seq} would be split.
  Therefore, $|M \oplus \coker f| > |M|$ holds, and hence $M \oplus \coker f$ is tilting. Now, since $M$ has two tilting complements and $\coker f \in \Fac M$, Theorem \ref{thm:exists-exchange-seq} implies that there is a tilting exchange sequence from $M \oplus X$.
  Moreover, since the left map of the tilting exchange sequence is a left minimal $\add M$-approximation, it should coincide with $f$. Thus \eqref{eq:tilt-mut-seq} is a tilting exchange sequence.
\end{proof}

\begin{ack}
  The authors would like to thank Osamu Iyama for helpful discussions.
  The first author is supported by JSPS KAKENHI Grant Number JP21J00299.
  The second author is supported by JSPS KAKENHI Grant Number JP22J20611.
\end{ack}

\end{document}